\documentclass[11pt]{amsart}
\usepackage{amsmath, amssymb,mathabx}
\usepackage[alphabetic]{amsrefs}
\usepackage{ xypic ,color}
\theoremstyle{definition}

\numberwithin{equation}{section}

\newcommand\Spec{\operatorname{Spec}}

\newcommand{\op}{\operatorname}
\newcommand\Pic{\operatorname{Pic}}

\newcommand\Ext{\operatorname{Ext}}

\newcommand\tensor{\otimes}
\newcommand\ml{\mathcal{L}}

\newcommand\mh{\mathcal{H}}

\newcommand\cur{C}
\newcommand\mq{\mathcal{Q}}
\newcommand\rk{\operatorname{rk}}

\newcommand\frg{\mathfrak{g}}

\newcommand{\leto}[1]{\stackrel{#1}{\to}}

\newcommand\Bun{\operatorname{Bun}}
\newcommand\Parbun{\operatorname{Parbun}}
\newcommand\LoG{\mathcal{L}\mathcal{G}}
\newcommand\Iwa{\mathcal{I}}

\newtheorem{theorem}{Theorem}[section]
\newtheorem{remark}[theorem]{ Remark}

\newtheorem{proposition}[theorem]{Proposition}

\newtheorem{lemma}[theorem]{Lemma}

\newtheorem{defi}[theorem]{Definition}

\begin{document}
\title[Principal bundles on singular curves] {Triviality properties of principal bundles on singular curves}
\author{Prakash Belkale and Najmuddin Fakhruddin}

\begin{abstract}
  We show that principal bundles for a semisimple group on an
  arbitrary affine curve over an algebraically closed field are
  trivial, provided the order of $\pi_1$ of the group is invertible in
  the ground field, or if the curve has semi-normal
  singularities. Several consequences and extensions of this result
  (and method) are given. As an application, we realize
conformal blocks bundles on moduli stacks of stable curves
 as push forwards of line bundles on
(relative) moduli stacks of principal bundles on the universal curve.
\end{abstract}
\maketitle
\section{Introduction}

It is a consequence of a theorem of Harder \cite[Satz 3.3]{harder1}
that generically trivial principal $G$-bundles on a smooth affine
curve $C$ over an arbitrary field $k$ are trivial if $G$ is a
semisimple and simply connected algebraic group. When $k$ is
algebraically closed and $G$ reductive, generic triviality,
conjectured by Serre, was proved by Steinberg\cite{steinberg} and
Borel--Springer \cite{BS}.

It follows that principal bundles for simply connected semisimple
groups over smooth affine curves over algebraically closed fields are
trivial.  This fact (and a generalization to families of bundles
\cite{DS}) plays an important role in the geometric realization of
conformal blocks for {\em smooth} curves as global sections of line
bundles on moduli-stacks of principal bundles on the curves (see the
review \cite{sorger} and the references therein).

An earlier result of Serre \cite[Th\'eor\`eme 1]{serre} (also
see \cite[Theorem 2]{atiyah})  implies that this triviality property
is true if $G=\operatorname{SL}(r)$, and $C$ is a possibly singular affine
curve over an arbitrary field $k$.
In \cite{BG} it was shown by a versal deformation argument that if $X$ is a
reduced projective curve with at worst nodal singularities over
$\Bbb{C}$ with $p_1,\dots,p_n$ in the smooth locus of $X$ so that
$\mathcal{O}(\sum p_i)$ is ample, there exists a dense open substack
of the moduli of $G$-bundles on $X$ such that any $E$ in this open
substack restricts to a trivial bundle on $X\setminus
\{p_1,\dots,p_n\}$.

The results of this paper show that the triviality properties of
principal bundles on arbitrary singular curves (in particular,
degenerating families of smooth curves) are very similar to those on
smooth curves. They allow us to realize
conformal blocks on moduli stacks of stable curves
$\overline{\mathcal{M}}_{g,n}$ as push forwards of line bundles on
(relative) moduli stacks of principal bundles on the universal curve (Theorem \ref {confreal}).

Alternate compactifications of $\mathcal{M}_{g,n}$
have been considered recently (cf. Hassett-Keel program \cite{Fedo});
triviality statements for  $G$ bundles over arbitrary affine curves
could potentially be useful in a geometric theory of conformal blocks
over such spaces.

\begin{theorem}\label{triv}
  Let $C$ be an arbitrary (possibly non-reduced, reducible or
  disconnected) affine curve over an algebraically closed field
  $k$. If $G$ is a semisimple algebraic group over $k$ such that
  $|\pi_1(G)|$ is invertible in $k$, then any principal $G$-bundle $E$
  on $C$ is trivial, i.e., there is a section $s:C\to E$.
\end{theorem}
The main idea is, assuming $C$ to be reduced, to produce a section
over the normalization of $C$ (using \cite{harder1}), which on the
inverse image of a (suitable) infinitesimal neighborhood $C_m$ of the
singular locus of $C$, agrees with the pull back of a section of $E$
over $C_m$.  Such sections are shown to descend to $C$.

The conditions for  triviality in Theorem \ref{triv} are
necessary and sufficient for arbitrary singularities, but for
semi-normal curves a stronger result holds, see Section
\ref{finalproblem}.

\begin{theorem}\label{Corollary1}
  Let $C$ be an arbitrary separated
   curve over an algebraically
  closed field $k$, and $E$ a principal $G$-bundle on $C$ for a
  connected reductive group $G$ over $k$.
\begin{enumerate}
\item[(a)] Let $Z$ be a finite subset of $C$. Then, there exists a
  Zariski open subset $U$ of $C$ containing $Z$ such that $E$ is
  trivial on $U$.
\item[(b)]  The structure group of $E$ can be reduced to $B$, where $B$ is a
  Borel subgroup of $G$.
\end{enumerate}
\end{theorem}
Note that Theorem \ref{Corollary1} in the case of generically reduced
$C$ follows easily from Theorem \ref{triv} if the order of $\pi_1$ of
the semisimple quotient of $G$ is invertible in $k$, but there are no
conditions on the characteristic of $k$, or on $C$ in Theorem
\ref{Corollary1}.

 A result of Bia{\l}ynicki-Birula \cite[Theorem 1]{BB} implies
 that  if $C$ is {irreducible}, then any $E$ as in Theorem \ref{Corollary1}
 is Zariski locally trivial. However, Zariski local triviality does not seem to
imply existence of $B$-structures (i.e., that the structure group of $E$ can be reduced to $B$) if the curve has more than  one singular point.

We also prove  versions of Theorems \ref{triv} and
\ref{Corollary1} which allow for families of non-constant curves (Theorems \ref{DS1}, \ref{triv3} and \ref{triv4} below).
Using Theorem \ref{Corollary1}, (b) as an input, these follow, very
closely, arguments of Drinfeld and Simpson \cite{DS} for similar
results for families of smooth curves:

Let $S$ be an arbitrary scheme over $\operatorname{Spec}(\Bbb{Z})$ and
let $f:X\to S$ be a proper, flat and finitely presented curve over
$S$. Let $G$ be a split\footnote{As is well known, any reductive group
  scheme becomes split after a surjective \'{e}tale base change, and so the
  splitness assumption is not required in Theorems
  \ref{DS1},\ \ref{triv3} and \ref{triv4}.} reductive group scheme over
$\operatorname{Spec}(\Bbb{Z})$ (base changed to $S$), and $B$ a Borel
subgroup of $G$.

\begin{theorem}\label{DS1}
  Let $E$ be a principal $G$-bundle on $X$ with $G$ connected and
  reductive. Then, after a surjective \'{e}tale base change, $S'\to
  S$, the structure group of $E$ can be reduced to $B$ (and hence $E$ becomes Zariski locally
  trivial).
\end{theorem}

Let $D\subset X$ be a
relatively ample effective Cartier divisor which is flat over $S$,
 and let
${U}=X\setminus D$.
\begin{theorem}\label{triv3}
  Let $E$ be a principal $G$-bundle on $X$ with $G$ semisimple and
  simply connected. Then, after a surjective \'{e}tale base change
  $S'\to S$, $E$ is trivial on $U_{S'}$.
\end{theorem}

For $G=\operatorname{SL}(n)$, Zariski localization is sufficient in
Theorem \ref{triv3} in many cases, see Remark \ref{subtle}.

If $G$ is not simply connected, we need further hypotheses,
having to do with existence of relative Picard schemes, reductions to
$\Pic^0$, and the nature of $\Pic^0$ of fibers
while generalizing arguments in \cite{DS}.

\begin{theorem}\label{triv4}
  Let $f:X\to S$ be a proper, flat and finitely presented curve, with
  $f$ cohomologically flat in dimension zero (see \cite[p. 259]{BLR}),
  $D\subset X$ a relatively ample and flat (over $S$) effective
  Cartier divisor, and $E$ a principal $G$-bundle on $X$.  Assume that
\begin{enumerate}
\item[(A)] \'{E}tale locally on $S$, $D\subset X$ is, set
  theoretically a union of sections (possibly not disjoint) of $f$,
\item[(B)] The morphism $f$ is smooth in a neighborhood of $D$,
\item[(C)] $G$ is semisimple, with $|\pi_1(G)|$ invertible in
  $\mathcal{O}_S$.
\end{enumerate}
Then, after a surjective \'{e}tale base change $S'\to S$, $E$ is
trivial on $U_{S'}$, where $U=X\setminus D$.
\end{theorem}
Note that the cohomological flatness condition on $f$ holds if it has
reduced geometric fibers.  We discuss a variant of this theorem in
Section \ref{finalproblem}.

The methods used in Theorem \ref{triv} can be used in the study of
questions related to the Grothendieck-Serre conjecture:
\begin{theorem}\label{GS}
  Let $X$ be a reduced surface over an algebraically closed field
  whose normalization is smooth. Let $G$ be a connected reductive group
  and $E$ a principal $G$-bundle on $X$ which is generically
  trivial. Then $E$ is locally trivial in the Zariski topology.
\end{theorem}
Note that there exist examples of principal bundles over normal
surfaces which are generically trivial but not locally trivial, see
Section \ref{exemple}. Over non-algebraically closed fields there
exist such bundles even over (singular) curves \cite{AG}.

In Section \ref{ind}, we prove the following: The irreducibility of the moduli stack of
$G$-bundles on a singular projective curve when $G$ is semisimple and
simply connected (Proposition \ref{irred}), an extension of the
uniformization theorem for $G$-bundles \cite{BL1,BL2,LS,DS} for singular curves (Proposition
\ref{uniformo}), and the integrality of the space of maps from a
reduced affine curve to $G$ when the base field is $\Bbb{C}$ and the
group $G$ is semisimple and simply connected (Proposition
\ref{LSgen}). The proof of Proposition
\ref{LSgen} uses Theorem \ref{triv3}, and
follows closely a proof of a similar result of Laszlo and Sorger
\cite{LS}. A new group theoretic input here is work on subgroups of
split simply connected semisimple groups generated by elementary
matrices \cite{IM,stein} over semilocal rings.

Recall that associated to a simple Lie algebra $\frg$ over $k=\Bbb{C}$, and dominant integral weights $\lambda_1,\dots,\lambda_n$ at a  level $\ell$ (see Section \ref{deficb}),
the theory of conformal blocks produces vector bundles of conformal blocks $\Bbb{V}_{\frg,\lambda,\ell}$ on the moduli stacks of stable $n$-pointed curves
$\overline{\mathcal{M}}_{g,n}$. Work in the 90's due to several authors \cite{BL1,Faltings,KNR} led to a realization of the duals of fibers of  $\Bbb{V}_{\frg,\lambda,\ell}$ over $\mathcal{M}_{g,n}$
as global sections of line bundles over suitable moduli spaces and stacks. Using Theorem \ref{triv3}, and the work of Beauville, Laszlo and Sorger (see the reviews \cite{sorger, sorger3}), we extend the stack theoretic realization of conformal blocks to all of $\overline{\mathcal{M}}_{g,n}$: Let $G$ be a simple, simply connected, complex algebraic group with Lie algebra $\frg$. In Section \ref{c11}, we consider a relative smooth Artin stack of parabolic bundles
$$\pi:\Parbun_{G,g,n}\to \overline{\mathcal{M}}_{g,n},$$ construct a line bundle $\ml$ on $\Parbun_{G,g,n}$, and obtain:
\begin{theorem}\label{confreal}
There is a canonical isomorphism
$\pi_*{\ml}\leto{\sim}\Bbb{V}_{\frg,\lambda,\ell}^*.$
\end{theorem}
Such isomorphisms are produced for any family of stable $n$-pointed curves. A proof of the above statement for a fixed singular stable pointed curve  also appears in \cite{BG} (which uses a different method).

For classical groups (and even levels in the case of Spin groups), the line bundle $\ml$ can be constructed explicitly in terms of determinant of cohomology (see Theorem \ref{confprime}). Finally, Picard groups of moduli stacks of parabolic bundles on (arbitarily) singular projective curves are computed in Section \ref{c2}.

\subsection{Acknowledgements}
We thank Yogish Holla for useful discussions and the referee for some
helpful comments.  This work was begun when P.B. was visiting TIFR
Mumbai in July--August 2015. P.B. thanks N.F. for the invitation, and
TIFR for its hospitality.

\section{Proof of Theorems 1.1 and 1.2}
We will first prove Theorem \ref{triv} under the assumption that $C$
is reduced, and then prove the case of arbitrary $C$; until further
notice, we assume that $C$ is reduced. Let $\pi:\widetilde{C}\to C$ be
its normalization, let $G$ be a semisimple algebraic group, and $E$ a
principal $G$-bundle on $C$. We assume that $|\pi_1(G)|$ is invertible
in $k$.

Let $R\subset C$ be a finite reduced subscheme such that $\pi$ is an
isomorphism over $C\setminus R$. Let $C_m$ be the $m$th
infinitesimal neighborhood of $R$ in $C$, $m\geq 0$. Let
$\widetilde{C}_m =\pi^{-1}(C_m)$.
\begin{lemma}\label{df}
  There exists $m\geq 0$ with the following property: $f\in
  H^0(\widetilde{C},\mathcal{O}_{\widetilde{C}})$ is in the image of
  $H^0({C},\mathcal{O}_{{C}})$ if and only if the restriction of $f$
  in $H^0(\widetilde{C}_m,\mathcal{O}_{\widetilde{C}_m})$ is in the
  image of $H^0({C},\mathcal{O}_{{C}})$ (equivalently $H^0({C}_m,\mathcal{O}_{{C}_m})$).
\end{lemma}
\begin{proof}
  $M=H^0(\widetilde{C},\mathcal{O}_{\widetilde{C}})/H^0({C},\mathcal{O}_{{C}})$
  is a finite $A=H^0({C},\mathcal{O}_{{C}})$ module supported on $R$,
  and is hence annihilated by a power $I^m$ of the ideal $I$ of $R$ in
  $C$, hence $M\leto{\sim}M\tensor A/I^m$. Tensor the exact sequence
$$0\to H^0({C},\mathcal{O}_{{C}})\to H^0(\widetilde{C},\mathcal{O}_{\widetilde{C}})\to M\to 0$$
by $A/I^m$, to get the right exact sequence
$$H^0({C},\mathcal{O}_{{C}_m})\to
H^0(\widetilde{C},\mathcal{O}_{\widetilde{C}_m})\to M\tensor A/I^m\to
0,$$ which implies the desired assertion.
\end{proof}
 Fix $m$ as in Lemma \ref{df}.  Suppose $E$ is a principal $G$-bundle
 on $C$.  Let $\widetilde{E}$ be the pull back $G$-bundle on
 $\widetilde{C}$.
 \begin{lemma} \label{sunday1} Any section $s_m$ of $\widetilde{E}$
   over $\widetilde{C}_m$ extends to a section of $\widetilde{E}$ over
   $\widetilde{C}$.
\end{lemma}

\begin{proof}
Since $\widetilde{E}$ is trivial as a $G$-bundle by \cite{harder1}
(see Remark \ref{harderR}), we only need to show that any
$\gamma_m:\widetilde{C}_m\to G$ extends to a map
$\gamma:\widetilde{C}\to G$, which is Lemma \ref{extension} below.
\end{proof}

\begin{lemma}\label{sunday2}
Suppose a section $s$ of $\widetilde{E}$ on $\widetilde{C}$  when restricted to $\widetilde{C}_m$ is the pull back of a section of $E$
  on $C_m$. Then $s$ is the pull back of a section of $E$ on $C$.
\end{lemma}
\begin{proof}
  Composing by the natural map $\widetilde{E}\to E$, we have a map
  $s':\widetilde{C}\to E$ which when restricted to $\widetilde{C}_m$
  is a composition
 $$\widetilde{C}_m\to C_m\leto{\alpha_m} E.$$
 We need to show that  ${s}'$ is itself a composition
 $$\widetilde{C}\to C\leto{\alpha} E.$$
 This assertion is local on $E$ (we may find an affine open subset of
 $E$ which contains the image of $\widetilde{C}_m$). Arguing
 coordinate by coordinate using Lemma \ref{df}, we see that $s'$
 descends to a map $C\to E$, since the restrictions to
 $\widetilde{C}_m$ come from functions on $C_m$ (and any relation
 between coordinate functions which holds on $\widetilde{C}$ holds on
 $C$ as well since functions on $C$ embed into functions on
 $\widetilde{C}$) .
\end{proof}
{{\bf Proof of of Theorem \ref{triv} in the case $C$ is reduced:}}
There is a section of $E$ over the reduced scheme $R$ (defined above)
since $R$
consists of finitely many points, and $k$ is algebraically closed.
Since $E\to C$ is smooth and $C_m\supset R$ is a nilpotent extension
of $R$, there exists a section $C_m\to E$ over $C_m$ extending this
section. Let $s_m$ be the induced section of $\widetilde{E}$ over
$\widetilde{C}_m$. We extend $s_m$ to a section $s$ of $\widetilde{E}$
over $\widetilde{C}$ using Lemma \ref{sunday1}. By Lemma \ref{sunday2}
this section descends to $C$. Therefore $E$ has a section over $C$ and
is hence trivial. \qed

\begin{remark}\label{harderR}
  Harder \cite{harder1} assumes that $G$ is simply connected and
  therefore any principal $G$-bundle on a smooth affine curve (over an
  algebraically closed field) is trivial if $G$ is semisimple and
  simply connected.  If $G$ is not simply connected, but still
  semisimple, let $\widetilde{G}$ be the simply connected cover of
  $G$. If $|\pi_1(G)|$ is invertible in $k$, $\tau:\widetilde{G}\to G$
  is \'{e}tale with finite kernel, and the map
  $H^1_{et}(C,\widetilde{G})\to H^1_{et}(C,G)$ is surjective (since
  $H^2_{et}(C,\ker(\tau))=0$).  Therefore any principal $G$-bundle on
  $C$ comes from a principal $\widetilde{G}$-bundle, and hence is
  trivial on $C$ (this reasoning is contained in \cite[Satz 3.3]
  {harder1}).
\end{remark}

\begin{lemma}\label{extension}
  Suppose $|\pi_1(G)|$ is invertible in $k$. Any morphism
  $\gamma_m:\widetilde{C}_m\to G$ extends to a morphism
  $\gamma:\widetilde{C}\to G$ .
\end{lemma}
\begin{proof}
  Let $\widetilde{G}$ be the simply connected form of $G$. Since
  $|\pi_1(G)|$ is invertible in $k$, $\widetilde{G}\to G$ is
  \'{e}tale. Thus, since $k$ is algebraically closed, by Lemma
  \ref{lift} (applied to $\widetilde{G}$, and to the residue fields of all points in $\widetilde{C}_m$, to get a morphism of $k$-schemes $(\widetilde{C}_m)_{\operatorname{red}}\to V$, further liftings of
  $(\widetilde{C}_m)\to V$ are constructed out of smoothness of $V\to G$) and the infinitesimal
  lifting property for smooth morphisms, there exists a morphism
  $\tau: \mathbb{A}_k^n \to G$ so that $\gamma_m$ lifts to a morphism
  $\gamma_m': \widetilde{C}_m \to \mathbb{A}_k^n$.

  By Lemma \ref{extensio} below,
   we can extend $\gamma'_m$ to a map
  $\gamma':\widetilde{C}\to \mathbb{A}_k^n$. The desired $\gamma$ is
  then $\tau\circ\gamma'$.
\end{proof}



\begin{lemma}\label{lift}
  Let $k$ be an arbitrary field and $G$ a simply connected, semisimple
  and split group over $k$. There exist a morphism $\eta: \mathbb{A}_k^n
  \to G$ and an open subset $V \subset \mathbb{A}_k^n$ such that $\eta
  |_V$ is smooth and for any extension $K$ of $k$ the map $V(K) \to
  G(K)$ induced by $\eta$ is surjective.
\end{lemma}

\begin{proof}
  Choose a collection of 1-parameter unipotent subgroups
  $U_i\leto{\sim}\mathbb{G}_a\subseteq \widetilde{G},i=1,\dots,s$,
  whose tangent spaces at the origin span\footnote{Simple coroots are
    in the span: This follows from the (easy) case of
    $\operatorname{SL}(2)$.}  the Lie algebra of
  $\widetilde{G}$. Consider the multiplication map
$$\eta:U := U_1\times_k U_2\times_k\dots\times_k U_s\to \widetilde{G}\to G.$$
By construction, $\eta$ is surjective on tangent spaces at the
identity $e$ of $U$.

Since $G$ is simply connected, it is well known (see,
e.g.~\cite{steinberg2}) that there exists $U'$, a product of
1-parameter unipotent groups as above, and $\eta':U' \to G$ so that
the induced map $U'(K) \to G(K)$ is surjective for all $K/k$. Then $U
\times U' \leto {\sim} \mathbb{A}_k^n$ for some $t $ and the map $U
\times U' \to G$ given by $\eta \cdot \eta'$ has the desired
properties: if $\eta'(x) = y$ for $x \in U'(K)$ then $\eta(e)\eta'(x)
= y$ and $\eta\cdot \eta'$ is smooth at $(e,x)$.
\end{proof}
\begin{lemma}\label{extensio}
  Let $T\subseteq X$ be a closed subscheme of an affine scheme $X$ over an arbitrary field $k$, and
  $f:T\to \Bbb{A}_k^n$ a morphism. Then $f$ extends to a morphism $X\to
  \Bbb{A}_k^n$.
\end{lemma}
\begin{proof}
  Write $T$ as the spectrum of $A/I$ with $I$ an ideal of
  $A=H^0(X,\mathcal{O}_{X})$. The function $f$ corresponds to an
  $n$-tuple $(f_1,\dots,f_n)$ of elements of $A/I$, which can be
  lifted to an $n$-tuple of elements in $A$.
\end{proof}

\begin{remark} \label{c1} If $k$ is a finite field or the function
  field of a curve over an algebraically closed field then any
  principal $G$-bundle on a curve over $k$ is generically trivial if
  $G$ is simply connected, semisimple and split
  \cite{harder2,JHS}. Since Lemma \ref{lift} holds over an arbitrary
  field and principal $G$-bundles over $\operatorname{Spec}(k)$, $k$
  as above, are also trivial, the above proof of Theorem \ref{triv}
  also shows that principal $G$-bundles on affine curves over such
  fields are trivial if $G$ is simply connected, semisimple and split.
  (This also holds for most groups $G$ as above if $k$ has
  characteristic zero and is of cohomological dimension one (e.g., a
  $C_1$-field): the function field of any curve over $k$ is perfect of
  cohomological dimension two and a conjecture of Serre, known for all
  classical groups as well as some exceptional groups (see
  \cite{BFP}), implies that any principal $G$-bundle (with $G$ as
  above) on a curve over $k$ is generically trivial.)
\end{remark}

{{\bf The general case of Theorem \ref{triv}}:} If $C$ is not reduced let $C_{\operatorname{red}}\subseteq C$ be the
reduced subscheme. Let $E'$ be the pull back principal $G$-bundle on
$C_{\operatorname{red}}$. By the case of Theorem \ref{triv} for reduced curves proved above, $E'$ is trivial, and hence we obtain a section $C_{\operatorname{red}}\to E'$.
Composing by the natural map $E'\to E$, we obtain a map $C_{\operatorname{red}}\to E$ which when composed with $E\to C$ gives the natural inclusion $C_{\operatorname{red}}\subseteq C$.

Since $E\to C$ is smooth, $C$ is affine, and
$C_{\operatorname{red}} \subseteq C$ is given by a nilpotent ideal, by
the infinitesimal criterion for smoothness we can lift
$C_{\operatorname{red}}\to E$ (over $C_{\operatorname{red}}\to C$) to
a map $C\to E$, such that the composite $C\to E\to C$ is the identity
map, i.e., $s$ is a section, and hence $E$ is trivial.

\subsection{Proof of Theorem 1.2}\label{proofcor1}


Since $B$-bundles are trivial in suitable Zariski neighbourhoods of any
given finite subset, (a) follows from (b).  Note that it suffices to
prove (a) when $C$ is reduced (trivializations on
$U_{\operatorname{red}}$ extend to $U$ when $U$ is affine).

\subsubsection{}
For (b) we start with the case when $C$ is affine: We need to produce
a section of $E/B$ over $C$.  We do this by replacing $E$ by $E/B$ and
$G$ by $G/B$ throughout in the proof of Theorem \ref{triv}. Here we
note that any $G$-bundle on $\widetilde{C}$ has a $B$-structure, and
there is a $B$-structure on $E$ restricted to any $C_m$ (when $C$ is reduced), since $E/B$
is smooth over $C$.  We then need to prove (the analogue of
Lemma \ref{extension}) that any function $\bar{f}:\widetilde{C}_m\to
G/B$ extends to a function $\widetilde{C}\to G/B$. This is easy
because $G/B$ is covered by affine spaces (i.e., of the form
$\Bbb{A}_k^m$): we can use the group $G$ to move the image of $\bar{f}$
into an affine space and then apply Lemma \ref{extensio}.

\subsubsection{}\label{abovee}
Clearly the case of affine $C$ implies (b) for all reduced $C$
(actually, all generically reduced $C$), because we can find an open
affine subset of our curve which contains all singularities, and a $B$-reduction of $E$
 on this open set, and then extend using the valuative criterion for properness.

\subsubsection{}
For the general case of (b) ($C$ possibly not generically reduced),
since we know $E$ is generically trivial, we may extend $E$ to a
compactification of $C$. Therefore assume $C$ is projective. By the
case of reduced curves already considered above, we know that $E$ has
a $B$-reduction on $Z=C_{\operatorname{red}}$. We now use methods from
\cite{DS}.

\begin{defi}\label{definitio}
 Let $\bar{\alpha}:B\to \Bbb{G}_m$ be the morphism associated to a positive
root $\alpha$ of $G$. If $E_B$ is a principal $B$-bundle on  $C$, let $E_{\alpha}$
be the line bundle on $C$ induced by $\bar{\alpha}$.
\end{defi}

Choose using Lemma \ref{big}, a $B$-reduction of $E$ on $Z$ such that
for every  positive root $\alpha$, $E_{\alpha}$ has sufficiently
large degrees (to be made clear below)  on irreducible components
of $Z$. We consider the problem of lifting this $B$-reduction through
nilpotent thickenings with square zero:
\begin{itemize}
\item Let $Z=C_{\operatorname{red}}\subset C_1\subset C_2\subset C$ be
  nilpotent thickenings such that the ideal $\mathcal{I}$ of $C_1$
  in $C_2$ has square zero.
\item Also assume that we have lifted our $B$-reduction on $E$ over
  $Z$ to a $B$-reduction $\sigma_1:C_1\to E/B$ over $C_1$.
\end{itemize}
Let $i_1:C_1\to C$ the inclusion and $\Theta$ the relative tangent
bundle of $E/B$ over $C$. The obstruction for lifting the reduction
$\sigma_1$ further to $C_2$ lies in $H^1(C_1, \sigma_1^*\Theta\tensor
\mathcal{I})$. Since we can filter $\sigma_1^*\Theta$ by line bundles
$i_1^*E_{\alpha}$, we are reduced to showing that $H^1(C_1,
i_1^*E_{\alpha}\tensor \mathcal{I})=0$ for all positive roots
$\alpha$. This is true because of the following (applied to $X=C_1$), and therefore
we can extend the $B$-reduction all the way to $C$, as desired:
\begin{lemma}\label{uniformV}
  Let $\mathcal{F}$ be a coherent sheaf on a projective curve
  $X$ over an algebraically closed field. There is a positive integer $N=N(\mathcal{F})$, such that if
  $Z=X_{\operatorname{red}}$, $Z=\cup Z_i$ with irreducible components
  $Z_i$, and $\mathcal{L}$ is a line bundle on $X$, then
$$\deg_{Z_i}(\mathcal{L}\mid_{Z_i})>N,\ i=1,\dots,s,
\implies H^1(X,\mathcal{L}\tensor \mathcal{F})=0.$$
\end{lemma}
\begin{proof}
  Let $\mathcal{J}$ be the ideal of $Z$ in $X$. We may filter
  $\mathcal{F}$ by sheaves $
  \mathcal{J}^s\mathcal{F}/\mathcal{J}^{s+1}\mathcal{F}$, and  may hence assume that
  $\mathcal{F}$ is a coherent sheaf on
  $Z$. Therefore we reduce to the case $X$ is reduced, which is standard.
\end{proof}

\begin{lemma}\label{big}
  Let $E$ be a principal $G$ bundle on a reduced projective curve
  $Y$ over an algebraically closed field. Assume that $E$ has a $B$-reduction. Then, $E$ has a $B$-reduction such that for every positive root $\alpha$ the degree of
  the corresponding $E_{\alpha}$ on each irreducible component of $Y$
  is at least $N$.
\end{lemma}
\begin{proof}
Each step in the proof of \cite[Proposition 3]{DS} generalizes:
\begin{enumerate}
\item By Theorem \ref{Corollary1} (a) in case of $C$ reduced, we can find an open affine  $U\subset Y$ that contain all
  singular points of $C$ such that $E$ is trivial on $U$.
  We may now replace $E$ by any other $E'$ which agrees with $E$ on $U$ (i.e., is trivial on $U$) provided
  we change $N$ to a suitable $N'$ (there is a bijection between
  $B$-reductions of $E$ and $E'$ with bounded differences of degrees of
  $E_{\alpha}$ and $E'_{\alpha}$).
\item We can therefore assume that  $E$ is  trivial.
\item Choose a finite morphism $Y\to \Bbb{P}^1$.
\item We are now reduced to the case of $Y=\Bbb{P}^1$, which is the
  same as in \cite{DS}.
\end{enumerate}
\end{proof}
In fact we have shown,
\begin{proposition}\label{newthing}
  Let $C$ be a projective curve over $k$, and $E$ a principal $G$-bundle.
   Then $E$ has a $B$-reduction $\sigma:C\to E/B$ such that
  $H^1(C,\sigma^*\Theta)=0$ where $\Theta$ is the relative tangent
  bundle of $E/B$ over $C$.
\end{proposition}
\begin{proof}
  By Lemma \ref{uniformV}, it would suffice if $E_{\alpha}$ have
  sufficiently large degrees on irreducible components of
  $C_{\operatorname{red}}$.  This follows from Lemma \ref{big} and
  the construction above which extends a $B$-reduction on $C_{\operatorname{red}}$
  with sufficiently large degrees on irreducible components to one on
  $C$.
\end{proof}

\section{Proofs of Theorems 1.3, 1.4 and 1.5}\label{laterr}
In these proofs we follow \cite{DS}: Using Theorem \ref{Corollary1}(b)
as a basic input, the arguments of \cite{DS} carry over with obvious
modifications.

\subsection{Proof of Theorem 1.3}
Let $F$ be the moduli functor of unobstructed $B$-reductions of $E$:
For a scheme $T$ over $S$, $F(T)$ is the space of $B$-reductions of
$E\times_S T$ { over } $X\times_S T$ such that for all $t\in T$,
\begin{equation}\label{vanish}
H^1(X_t,\sigma^*\Theta)=0
\end{equation}
where $\Theta$ is the relative tangent bundle of $E/B$ over $X$, and
$\sigma:X_t\to X$.

It follows from the theory of Hilbert schemes and deformation theory
that $F$ is representable by a scheme $\phi:M\to S$, with $\phi$
smooth. To prove Theorem \ref{DS1}, it suffices to show that any $s
\in S$ is in $\phi(M)$. This is because, by the smoothness of $\phi$,
for any $s \in \phi(M)$ we can find an \'{e}tale neighbourhood $S'\to
S$ of $s$ such that there is a section $S'\to M$ (over $S$).

To see that $s \in \phi(M)$ we may base change to the algebraic
closure $k$ of $k(s)$. By Theorem\ref{Corollary1} (b), and Proposition
\ref{newthing}, there is a $B$-reduction on $X_s\times_{k(s)} k$ with
the desired vanishing property \eqref{vanish}.

\subsection{Proof of Theorems 1.4 and 1.5}\label{DSgen}


We will prove \ref{triv4} following the proof of \cite[Theorem
3]{DS}. The proof of Theorem \ref{triv3} is similar, except that we do
not need the ``reduction to simply connected $G$" part. In the
following proof $S$ is always assumed to be affine, and hence $U$ is
also affine.

By Theorem \ref{DS1}, we find a $B$-reduction of $E$ after passing to
an \'{e}tale cover. Since $U$ is affine, the structure group of a
$B$-bundle can be reduced to a maximal torus $H$: in fact if $E$ is a
$B$-bundle on $X$, $E_H$ the corresponding $H$-bundle, then $E$ and
the $B$-bundle induced from $H$ via $H\to B$ are isomorphic over
$U$. Therefore we can assume that our principal bundle $E$ on $X$
comes from an $H$-bundle. The final step is then to show that any
$H$-bundle on $X$ becomes trivial on $U$ after extension of structure
group to $G$ (and \'{e}tale base change).


\subsubsection{Reduction to the case of simply connected $G$}

The arguments in \cite[Section 6, second paragraph]{DS} for passage
from $G$ to its simply connected cover $\widetilde{G}$ can be broken
up into two parts. The first is reduction to $\Pic^0$. For this we
note that cohomological flatness in dimension $0$ implies, by a
theorem of M.~Artin, that $\Pic_{X/S}$ exists as an algebraic space
over $S$, \cite[Theorem 7.3]{ART}, \cite[Theorem 1, p. 223]{BLR};
since any algebraic space has an \'{e}tale covering by a scheme, this
does not create any extra difficulty.  The main idea in the reduction
is then to modify an $H$-bundle on $X$, keeping it unchanged over $U$,
such that for every $\tau\in A=\operatorname{Hom}(H,\Bbb{G}_m)$ the
induced line bundle is in $\Pic^0$ of each fiber.  Assumptions (A) and
(B) in Theorem \ref{triv4} allow us to make such a modification. This
is because they imply that given any line bundle $L$ on $X$ we may,
after an \'{e}tale base change, find a Cartier divisor $D'$ on $X$,
with support contained in the support of $D$, so that $L(D')$ has
degree $0$ on each irreducible component of each fiber.

The second part concerns the natural homomorphism
$T:\operatorname{Hom}(\widetilde{A},\Pic^0)\to \operatorname{Hom}
(A,\Pic^0)$, where $\widetilde{A}=\operatorname{Hom}
(\widetilde{H},\Bbb{G}_m)$ ($\widetilde{H}$ is the maximal torus of
$\widetilde{G}$ over $H$). Since $A$ is a subgroup of $\widetilde{A}$
of index $|\pi_1(G)|$, $T$ is a morphism of group schemes over $S$
which, if the arithmetic genus of the fibres is positive, is \'{e}tale
iff $|\pi_1(G)|$ is invertible in $\mathcal{O}_S$.

Now given a $G$-bundle $E$ on $X$, using Theorem \ref{triv3} we may
assume that it comes from a $B$-bundle $F$. On $U$, $E$ is isomorphic
to the bundle induced from $F$ via the maps $B \to H \to G$, so we may
assume that $E$ is induced from an $H$-bundle $E'$, and we obtain a
section of $\operatorname{Hom}(A,\Pic^0)$ over $S$. Form the fibre
product $S'$ of $S$ and $\operatorname{Hom}(\widetilde{A},\Pic^0)$
over $\operatorname{Hom}(A,\Pic^0)$.  Clearly $S'\to S$ is \'{e}tale
and surjective, and the pull back of $E'$ to $X\times_S S'$ can be
lifted to a $\widetilde{H}$-bundle locally with respect to the Zariski
topology of $S'$.  We may therefore assume that $E$ is induced from a
$\widetilde{G}$-bundle, and comes from a $\widetilde{H}$-bundle.


\begin{remark}\label{variation}
  If $|\pi_1(G)|$ is not invertible in $\mathcal{O}_S$, $T$ is flat if
  (and only if) $\Pic^0$ of geometric fibers are semi-abelian.
\end{remark}

\subsubsection{The case of simply connected $G$}\label{einfach}
Assume now that $G$ is simply connected and that $E$ comes from an
$H$-bundle.  Since $G$ is simply connected,
$\operatorname{Hom}(\Bbb{G}_m,H)$ is freely generated by simple
coroots. Therefore we are reduced to checking that if $H$-bundles
$E_1$ and $E_2$ differ by the image of some $\Bbb{G}_m$-bundle via a
coroot $\check{\alpha}:\Bbb{G}_m\to H$, then the $G$-bundles
corresponding to $E_1$ and $E_2$ are isomorphic on $U_S$, after
Zariski localization in $S$.

The following simple lemma is used without proof in \cite[p.~386]{DS};
we give one here  for the reader's convenience.
\begin{lemma} \label{lem:prod}
  Let $L\subseteq G$ be the subgroup generated by $H$ and
  $r(\operatorname{SL}(2))$, where $r:\operatorname{SL}(2)\to G$
  corresponds to $\alpha$ (so $\check{\alpha}$ factors through $r$ as
  in \cite[Theorem 1.2.7, Definition 1.2.8]{C}).  Then
  $L= \operatorname{SL}(2)\times T$ or
  $L=\operatorname{GL}(2)\times T'$ where $T$, $T'$ are subtori of
  $H$.  Furthermore, $\tilde{r}$, which is $r$ viewed as a map
  $\operatorname{SL}(2)\to L$, is given by $\tilde{r}(g)=g\times e$ in
  both cases.
\end{lemma}
\begin{proof}
We first note that $r$ is an isomorphism onto its image since $G$ is
simply connected; we use this to identify $r(\operatorname{SL}(2))$ with
$\operatorname{SL}(2)$.

Let $Z$ be the identity component of the (reduced) centre of $L$, so
$Z$ is a codimension one subtorus of $H$. The inclusions induce a
surjective morphism
$$\pi:\operatorname{SL}(2)\times Z\to L\ .$$
Let $K$ be the (scheme theoretic) intersection of $Z$ and
$\operatorname{SL}(2)$.  The (scheme theoretic) kernel of $\pi$ is
then equal to the image of $K$ embedded diagonally in
$\operatorname{SL}(2)\times Z$.

Thus, if $K$ is trivial then $\pi$ is an isomorphism so we get the
first possibility (take $T = Z$). Otherwise, we have $K \cong \mu_2$
since it is contained in the (scheme theoretic) centre of
$\operatorname{SL}(2)$ and is non-trivial. Write $Z$ as a product
$A\times T'$ where $A \cong\Bbb{G}_m$ contains $K$; that this is
always possible is easily seen using character groups. The subgroup
$M$ of $L$ generated by $A$ and $\operatorname{SL}(2)$ is isomorphic
to $\operatorname{GL}(2)$ since $A$ and $\operatorname{SL}(2)$ commute
and intersect along $K \cong \mu_2$. Finally, the scheme theoretic
intersection of $M$ and $T'$ is trivial, so the map $M\times T'\to L$
(induced by the inclusions) is an isomorphism.
\end{proof}

We now apply Lemma \ref{lem:prod}. Since $H\subseteq L$, $E_1$ and
$E_2$ both come from $L$-bundles. Let $D$ (resp. $D'$) denote the
maximal torus of $\operatorname{SL}(2)$
(resp. $\operatorname{GL}(2)$).  Using the form of $\tilde{r}$ in the
lemma, we may assume that $E_1$ and $E_2$ come from
$D\times T$-bundles (resp. $D'\times T'$-bundles), differing by a
$\Bbb{G}_m$-bundle via $\Bbb{G}_m\to D$ (resp. $\Bbb{G}_m\to D'$)
corresponding to the (standard) coroot of $\operatorname{SL}(2)$.

We may clearly ignore the $T$ and $T'$ factors, and need to show that the corresponding
$\operatorname{SL}(2)$ (resp. $\operatorname{GL}(2)$) components of $E_1$ and $E_2$ are isomorphic on $U_S$, after
Zariski localization in $S$. The \v{C}ech cocycles representing $E_1$ and $E_2$ in the $\operatorname{GL}(2)$ case
have ratios in $\operatorname{SL}(2)$, and therefore have isomorphic determinant line bundles when considered as rank two vector bundles.

Therefore, Theorem \ref{triv4} in the case of simply connected $G$ follows from the following assertion: Let $f:X\to S$ be a proper, flat and finitely presented curve over
$S$, and $D\subset X$  a
relatively ample effective Cartier divisor which is flat over $S$.
Suppose $E_1$ and $E_2$
are either: {\em (Case 1)} Principal $G=\operatorname{SL}(2)$-bundles on
$X$, or, {\em (Case 2)} Principal $G=\operatorname{GL}(2)$-bundles on
$X$ with the same determinant. Then,
\begin{proposition}\label{gl2sl2}
  $E_1$ and $E_2$ are isomorphic as $G$-bundles on $U$, after a surjective \'{e}tale base change  of $S$. Furthermore if $S$ is the spectrum of a separably closed field,
  $E_1$ and $E_2$ are in the same connected component of the moduli stack of $G$-bundles on
  $X$.
\end{proposition}
\begin{proof}
  Let $s\in S$, we will show that in an \'{e}tale  neighborhood of
  $s$, we have for sufficiently large $n$ exact sequences $i=1,2$,
\begin{equation}\label{exacto}
0\to \mathcal{O}\to E_i(nD)\to Q_i\to 0 \ .
\end{equation}
Because of our assumptions $Q_1$ and $Q_2$ are isomorphic and are
trivial on $U$ in case (1). Since $\mathcal{O}(D)$ is relatively
ample, the exact sequence \eqref{exacto} splits over $U$, with
$Q_i=\det E_i$.
The desired conclusions follow: for the deformation we
use the Ext-space $\Ext^1(Q_i,\mathcal{O})$.

We may assume that the residue field of $S$ at $s$ is infinite (in fact separably closed).
We produce such exact sequences by finding nowhere vanishing global
sections $\alpha_i\in H^0(X_s,{E}_i(nD))$, where $H^1(X_s,E_i(nD))=0$.
We choose $n$ sufficiently large so that $H^1(X_s,E_i(nD))=0$ and
$E_i(nD)$ are globally generated on $X_s$,\ $i=1,2$. Then the usual
Serre argument works: The subset of $H^0(X_s,{E}_i(nD))$ formed by
sections that vanish at a point $q$ has codimension $\rk E=2$. Taking
the union over all $q\in X_s$, the set of ``bad sections" lies on a
codimension one subvariety of the vector space $H^0(X_s,{E}_i(nD))$.
\end{proof}
\begin{remark}\label{subtle}
As we explain below, Zariski localization is sufficient  in Theorem \ref{triv3} for $G=\operatorname{SL}(n)$, and in Proposition \ref{gl2sl2} in each of the following two cases (a) The residue fields of $S$ are infinite, and (b) $X\to S$ is smooth in a neighbourhood of D.

In case (a) the modification of Proposition \ref{gl2sl2} follows by the same method, and for Theorem \ref{triv3},  the method of \cite[Theorem 2]{atiyah} and
\cite[Lemma 3.5]{BL1} applies without changes. Here we use the fact that the complement of any hyperpersurface in $\Bbb{A}^m_k$
has $k$-rational points when $k$ is infinite, so that dimension counting arguments apply.

In case (b), fix a point $s\in S$. To get the modification of Theorem \eqref{triv3}, we  find everywhere linearly independent sections $\alpha_1,\alpha_2$ of $E$ on $X_s- D_s$ by \cite[Th\'eor\`eme 1]{serre}. Here $E$ is a vector bundle on $X$ with trivialized determinant. A linear combination $\alpha_1 +f \alpha_2$ ($f$ a function on $U_s$) can be found with sufficiently high orders of poles at all points of $D_s$, and we obtain a subbundle $\mathcal{O}\subseteq E_s(nD_s)$ with $n$ sufficiently large. This can be deformed to a Zariski neighborhood of $s$.

In Proposition \ref{gl2sl2}, write (again using \cite[Th\'eor\`eme 1]{serre}) $E_{i}$ restricted to $U_s=X_s-D_s$ as $\mathcal{O}\alpha^{(i)}_1\oplus \ml_i$ where $\ml_i$ are
line bundles on $U_s$, and $\alpha^{(i)}_1$ are nowhere vanishing sections. Let $\alpha^{(i)}_2$ be sections of $\ml_i$  on $X_s-D_s$.  Sections of $(E_i)_s$ of the form $\alpha^{(i)}_1 + f^{(i)}\alpha^{(i)}_2$ can be found with sufficiently high orders of poles at all points of $D_s$, and these result in subbundles $\mathcal{O}\subseteq (E_i)_s(nD_s)$ on $X_s$ for $n$ sufficiently
large; these  can be deformed to a Zariski neighborhood of $s$.

\end{remark}

\subsection{Refinements of Theorems 1.1 and 1.5}\label{finalproblem}

In Theorem \ref{triv4}, we can replace condition (C) by (C'): $\Pic^0$
of geometric fibers are semi-abelian, and draw the weaker conclusion
that after a faithfully flat base change $S'\to S$ with $S'$ locally
of finite presentation over $S$, $E$ is trivial on $U_{S'}$, where
$U=X\setminus D$ (see Remark \ref{variation}). This generalizes the
flat base change part of \cite[Theorem 3]{DS} and also shows that
Theorem \ref{triv} holds for semi-normal curves without any condition
on $|\pi_1(G)|$.

Note that $\Pic^0$ of a projective curve being semi-abelian is
equivalent to the curve having geometrically semi-normal singularities
\cite[Chapter 9.2]{BLR}; in particular, as is well known, this
condition holds for semistable curves.

Now let $C$ be a (generically reduced) affine curve over an
algebraically closed field $k$ of characteristic $p >0$ with a
compactification $\overline{C}$ (which is smooth at the boundary) and
such that $\Pic^0(\overline{C})$ is not semi-abelian. We will show
that if $G$ is a semisimple group such that $p\divides |\pi_1(G)|$ then
there exist non-trivial $G$-bundles on $C$.

Let $\widetilde{G}$ be the simply connected cover of $G$ and let
$\widetilde{H}$ be the inverse image in $\widetilde{G}$ of a maximal
torus $H$ of $G$. Then $K=\ker(\widetilde{H} \to H) =
\ker(\widetilde{G} \to G)$ is a finite multiplicative group scheme of
order divisible by $p$.  $H$ and $\widetilde{H}$ are tori of the same
dimension, say $r$, so $H^1(\overline{C}, H)$ and $H^1(\overline{C},
\widetilde{H})$ are both isomorphic to $\Pic(\overline{C})^r$. Since
we have assumed that $p\divides |\pi_1(G)|$ and
$\Pic^0(\overline{C})$ is not semi-abelian, so contains $\mathbb{G}_a$
as a subgroup, the map
\[
H^1(\overline{C}, \widetilde{H}) \to H^1(\overline{C}, H)
\]
is not surjective and in fact has an infinitely generated cokernel
(since this holds for the multiplication by $p$ map on
$\Pic^0(\overline{C})$). Since $\overline{C}$ is smooth at the
boundary points, the surjective map $\Pic(\overline{C}) \to \Pic(C)$
has finitely generated kernel. It follows that the image of the second
map in the exact sequence
\[
H^1(C, \widetilde{H}) \to H^1(C, H) \to H^2_{fl}(C, K)
\]
is not finitely generated. If $E_H$ is any $H$-bundle on $C$ which
which does not lift to an $\widetilde{H}$-bundle, it follows that the
induced $G$-bundle $E$ has a non-zero class in $H^2_{fl}(C, K)$; in
particular, it is not trivial.

\section{Proof of Theorem 1.6}
We will in fact prove the following (stronger) statement:
\begin{theorem}\label{GSnew}
 Let $X$ be an
algebraic variety (of arbitrary dimension) over an algebraically
closed field, and $E$ a principal $G$-bundle on $X$ where $G$ is a
connected reductive group. Let $\pi:\widetilde{X}\to X$ be the
normalization. Assume that
\begin{enumerate}
\item[(a)] $\widetilde{X}$ is smooth.
\item[(b)] $E$ is generically trivial.
\item[(c)] Let $R\subset X$ be a reduced subscheme such that $\pi$ is
  an isomorphism over $U=X\setminus R$. Assume that $E$ restricted to
  $R$ is Zariski locally trivial (e.g., if $\dim R\leq 1$ (by Theorem
  \ref{Corollary1})).
\end{enumerate}
Then,
 $E$ is Zariski locally trivial.
\end{theorem}

By work on the Grothendieck-Serre conjecture \cite{PSV}, one knows
that since $\widetilde{E}$, the pull back of $E$ to $\widetilde{X}$,
is generically trivial (by assumption (b)), it is locally trivial in
the Zariski topology. In fact given a finite subset $Z\subset
\widetilde{X}$ there exists a Zariski open subset $V\subset
\widetilde{X}$ containing $Z$ over which $\widetilde{E}$ is trivial.
Let $x\in X$ and $Z=\pi^{-1}(x)$, a finite set. We need to produce an
open subset of $X$ containing $x$ over which $E$ has a $B$-structure,
where $B$ is a Borel subgroup of $G$.  Clearly, we may assume $x\in
R$.


Let $V\subset \widetilde{X}$ be an open set containing $Z$ over which
$\widetilde{E}$ is trivial.  Replacing $X$ by an open subset
containing $x$ in $X-\pi(\widetilde{X}\setminus V)$, we can assume that
$\widetilde{E}$ is trivial on $\widetilde{X}$, and that $X$ is affine.

Let $X_m$ be the $m$th infinitesimal neighbourhood of $R$, and
$\widetilde{X}_m=\pi^{-1} X_m$. We can now follow the same method of proof
as for Theorem \ref{triv}:

\begin{itemize}
\item Any section of $\widetilde{E}/B$ on $\widetilde{X}_m$ extends to any affine
  Zariski neighborhood of $Z$ in $\widetilde{X}$.
\item There exists a section of $E/B$ on any infinitesimal neighborhood
  $X_m$. This is because by assumption there is a section of $E/B$ over
  $R$ (after Zariski localization), and can be extended to
  infinitesimal neighborhoods of $R$. This section can be pulled up to
  $\widetilde{X}_m$.
\item For $m$ sufficiently large, a section $s$ of $\widetilde{E}/B$ over
  $\widetilde{X}$ descends to a section of $E/B$ over $X$ if the
  restriction of $s$ to $\widetilde{X}_m$ is a pull back of a section
  of $E/B$ over $X_m$ (the same proof as for Lemma \ref{extension}).
\end{itemize}

\subsection{An example}\label{exemple}

We give examples of normal surfaces $X$ over any algebraically closed
field $k$ on which there exist $PGL(m)$-bundles which are generically
trivial but not Zariski locally trivial.

We begin with some preliminary remarks: Firstly, to prove the
existence of $PGL(m)$-bundles as above (with $m$ not specified) it
suffices to construct elements of the Brauer group $Br(X)$ which are
generically trivial but not locally trivial. (That such examples
should not be difficult to construct has been suggested by
Grothendieck \cite[p.~75]{GBII}, but we do not know a reference where
this has been made explicit.)  Furthermore, by a theorem of Gabber
(\cite{gabber}, \cite[Corollary 9]{Hoobler}) for any normal
quasi-projective surface\footnote{In fact, Gabber has proved this for
  any quasi-projective variety (unpublished, see \cite{deJ}) and
  Schr\"oer has proved this for any geometrically normal (separated)
  surface.}  $X$ the Brauer group is equal to the torsion in
$H^2_{et}(X, \mathbb{G}_m)$, so if the characteristic $p$ of $k$ does
not divide $n$, the Kummer sequence shows that it suffices to find
elements of $H^2_{et}(X, \mu_n)$ which are generically trivial but not
locally trivial. We give an example where such elements exist for any
such $n>1$.

Let $E \subset \mathbb{P}^2_k$ be an elliptic curve and let $p_1, p_2,
\dots, p_{10}$ be distinct points on $E$. Let $X'$ be the blow up of
$\mathbb{P}^2_k$ at these points and let $E'$ be the strict transform
of $E$. We have $(E')^2 = -1$ and if there exists an ample line bundle
$L$ on $X'$ and a positive integer $a$ such that $L|_{E'} \cong
\mathcal{O}(-aE')|_{E'}$ then it is well known (and easy to see) that
there exists a morphism $\pi: X' \to X$, where $X$ is a normal
projective surface, which contracts $E'$ and is an isomorphism onto
its image when restricted to $X' \setminus E'$. To ensure this it
suffices, for example, to have positive integers $a', b'$ such that
$\mathcal{O}_{X'}(a'H)|_{E'} =\mathcal{O}_{E'}( b(p_1 +p_2 +\dots +
p_{10}))$ in $\Pic(E)$, where $H$ is the pullback of the class of a
line in $\mathbb{P}^2$.

We now use the Leray spectral sequence for the map $\pi$ and
cohomology with $\mu_n$ coefficients. By proper base change,
$R^1\pi_*\mu_n$ is a skyscraper sheaf with stalk $(\mathbb{Z}/n)^2$
supported on the singular point $p = \pi(E')$ of $X$. Since $X'$ is
simply connected $H^1_{et}(X, \mu_n) = 0$, so the differential $d_2$ of the
spectral sequence gives an embedding of $H^0(X, R^1\pi_*\mu_n)$ into
$H^2_{et}(X, \mu_n)$; denote this subgroup by $A$. It is clear that $A$
maps to $0$ in $H^2(X \setminus \{p\}, \mu_n)$ as well as in $H^2_{et}(X',
\mu_n)$. However, we shall show that for all but finitely many $n$ it
injects into $H^2_{et}(U, \mu_n)$, where $U$ is any Zariski open
neighbourhood of $p$.

Let $\alpha$ be any element of $A$. If $\alpha$ dies in $U$, the Gysin
sequence shows $\alpha$ must be the cohomology class of a divisor $D$
supported in $X \setminus U$. Let $U' = \pi^{-1}(U)$; we may view $D$
as a divisor on $X'$ since $X \setminus U$ is identified naturally
with $X' \setminus U'$. Since $X'$ is a smooth projective rational
surface, $H^2(X', \mathbb{Z}/n) = \Pic(X')/n\Pic(X')$, so it follows
that $D$ must be divisible by $n$ in $\Pic(X')$. The map $\pi^*: \Pic(X)
\to \Pic(X')$ is an injection and the cokernel is a finitely generated
abelian group. If $n$ is coprime to the order of the torsion part of
the cokernel, then $D$ must be divisible by $n$ in $\Pic(X)$, so
$\alpha$ must be $0$.

By choosing the $p_i$ suitably, one can arrange that the cokernel is
torsion free. For example, if $k$ is not $\overline{\mathbb{F}}_p$,
one can choose $p_1,p_2,\dots,p_9$ arbitrarily so that $H|_{E},
p_1,\dots,p_9$ are independent in $\Pic(E)$ and then choose $p_{10}$ to
satisfy the equation
$10H|_{E} = 3(p_1 + p_2 + \dots + p_{10}) $
in $\Pic(E)$. In general, one has $\Pic(X) = Ker(\Pic(X') \to \Pic(E'))$;
with the above choice one then sees that $\Pic(X)$ is of rank $1$ and
the cokernel is torsion free.

\begin{remark}
  Given an $\alpha$ as above of order $n$, we do not know what is the
  smallest integer $m$ so that $\alpha$ comes from a $PGL(m)$ torsor
  on $X$.
\end{remark}

\section{Applications}\label{ind}

\begin{proposition}\label{irred}
Let $\overline{C}$ be a  projective curve over an
algebraically closed field $k$, and $G$ a semisimple  simply connected group. The moduli stack
  $\operatorname{Bun}_G(\overline{C})$ of principal $G$-bundles on
  $\overline{C}$ is connected (it is well known to be smooth, see,
  e.g., \cite{wang}).
\end{proposition}
\begin{proof}
  The proof of Proposition 5 in
  \cite{DS} generalises easily (using  Theorem \ref{Corollary1})
  with one small change in detail.

  Let $B$ be a Borel subgroup of $G$, and $H\subset B$ a maximal
  torus.  Let $\Bun_G$, $\Bun_H$ and $\Bun_B$ denote the corresponding
  stacks of bundles on $\overline {C}$.  $\pi_0(\Bun_B)$ surjects onto
  $\pi_0(\Bun_G)$ (by Theorem \ref{Corollary1}), and $\pi_0(\Bun_H)$
  is in bijection with $\pi_0(\Bun_B)$ (as in \cite{DS}). Therefore it
  suffices to connect the image of an $H$-bundle in $\Bun_G$ to the
  trivial $G$-bundle. To do this we need to show (since coroots
  generate cocharacters) that if we twist an $H$-bundle by a
  cocharacter, the image of the $H$-bundle (in $\Bun_G$) stays in the
  same connected component. This follows from Proposition \ref{gl2sl2}
  (see Section \ref{einfach}).
\end{proof}

Assume for the rest of this section that $\overline{C}$ is a reduced
projective curve over $k=\Bbb{C}$. Let $p_1,\dots,p_n$ be smooth
points of $\overline{C}$ such that $C=\overline{C}-\{p_1,\dots,p_n\}$
is affine. Choose local uniformizing parameters at the points $p_i$.

Consider the $k$-groups $L_G^C$, $L_G$ and $L_G^+$ which assign to a
$k$-algebra $R$, the groups $G(\Gamma(C,\mathcal{O})\tensor R)$,
$G(R((z)))$ and $G(R[[z]])$ respectively. $L_G$ is an ind-scheme, and
$L_G^+$ is an affine scheme\footnote{An ind-scheme in this paper is a
  $k$-space (i.e., a set valued functor on the category of commutative
  $k$-algebras) of the form $\varinjlim Y_n$, indexed by the natural
  numbers, such that each $Y_n$ is a scheme over $k$ and all the maps
  $Y_n\to Y_{n+1}$ are closed embeddings.}.  There is an embedding
$L_G^C\to (L_G)^n$, and $L_G^C$ inherits an ind-scheme structure from
$(L_G)^n$.

As a corollary of Theorem \ref{triv4}, we obtain the following
generalization of the uniformization theorem \cite{BL1,BL2,LS,DS} for
singular curves:
Let $\mathcal{Q}_G= L_G/L_G^+$ be the affine Grassmannian.
\begin{proposition}\label{uniformo}
Suppose $G$ is semisimple (possibly not simply connected). There is
an isomorphism of stacks
$$L_G^C\setminus(\mq_G)^n\leto{\sim} \operatorname{Bun}_G(\overline{C}).$$
\end{proposition}
The morphism is constructed using \cite{BL2} (note that
$p_1,\dots,p_n$ are smooth points of $\overline{C}$).  That it is an
isomorphism follows from Theorem \ref{triv3}.

The following is a generalization of a result of Laszlo and Sorger
\cite[Proposition 5.1]{LS}:
\begin{proposition}\label{LSgen}Assume that $G$ is simple and simply
connected. The $k$-group $L_G^C$ is integral as an ind-scheme, i.e., can be
  written as a limit of an increasing sequence of integral schemes.
\end{proposition}
\begin{proof}
Following the proof of \cite{LS} (and Theorem \ref{triv3}), we see that $(\mathcal{Q}_G)^n\to
\operatorname{Bun}_G(\overline{C})$ is locally (in the \'{e}tale topology) a fiber bundle with
fiber $L_G^C$. This shows that $L_G^C$ is reduced. Connected
ind-groups are irreducible by a theorem of Shafarevich \cite{S},
therefore, it suffices (exactly as in \cite{LS}) to show that $L_G^C$
is connected.

We follow the proof in \cite{LS} of the connectedness of $L_G^C$ when $C$ is smooth
(attributed there to Drinfeld). Let $C'=C-\{p\}$ where $p$ is a smooth
point. Then,  using Theorem \ref{triv3}, $$L_G^{C'}/L_G^C\leto{\sim} \mq_G$$ exactly as in
\cite{LS} and
one obtains that the (sets of) connected components of $L_G^C$ and $L_G^{C'}$
are in bijection.  Therefore removing any finite number of smooth
points of $C$ does not change the number of connected components of
$L_G^C$.

Let $g\in L_G^C(k)$. Let $Z$ be the set of singularities of $C$, and
$K_Z$ the semilocal ring at $Z$: These are rational functions on $C$
which are regular at points of $Z$. Now $g$ gives rise to a point in
$G(K_T)$. It is known by the works \cite{IM,stein} (also see the
recent article \cite{SSV}), that $G(K_Z)$ is generated by $U^{+}(K_Z)$
and $U^-(K_Z)$, where $\delta_{\pm}:\Bbb{A}_k^m\leto{\sim} U^{\pm}$
are the unipotent radicals of a Borel $B^+$ and an opposite $B^{-}$.
Therefore we have maps $$u_i:C''\to \Bbb{A}_k^{m}, i=1,\dots,s$$ so
that $g=\delta_{\pm}u_1\dots \delta_{\pm}u_s$ and
$C''=C-\{q_1,\dots,q_m\}$ where $q_i$ are smooth points of $C$. Since
the sets of connected components of $L_G^C$ and $L_G^{C''}$ are in bijection, we just
need to connect $g$ and $1$ in $L_G^{C''}$: For this it suffices to consider the map
$\Bbb{A}_k^1\to L_G^{C''}$ given by $(t,x)\mapsto
\delta_{\pm}(tu_1(x))\dots \delta_{\pm}(tu_s(x)).$

\end{proof}

\section{Conformal blocks and generalized theta functions}\label{c11}
Let $\overline{\mathcal{M}}_{g,n}$ be the moduli stack parameterizing
stable $n$-pointed curves $(C;\vec{p})=(C,p_1,\dots,p_n)$ of genus
$g$. In this section $k=\Bbb{C}$.

\subsection{Conformal blocks}\label{deficb} Let $G$ be a simple,
connected, simply connected complex algebraic group with Lie algebra
$\frg$, $B$ a Borel subgroup of $G$ and $H \subset B$ a maximal torus.
For an integer (called the level) $\ell\geq 0$, let $P_{\ell}(\frg)$
denote the set of dominant integral weights $\lambda$ with
$(\lambda,\theta)\leq \ell$, where $\theta$ is the highest root, and
$(\ ,\ )$ is the Killing form, normalized so that
$(\theta,\theta)=2$. Let $\hat{\frg}$ be the Kac-Moody central
extension of $\frg$ (see e.g., \cite{sorger}). For a dominant integral
weight $\lambda$ in $P_{\ell}(\frg)$, let $\mathcal{H}_{\lambda,\ell}$
denote the corresponding irreducible representation of
$\hat{\frg}$. Note that
$V_{\lambda}\subseteq \mathcal{H}_{\lambda,\ell}$ where $V_{\lambda}$
is the corresponding irreducible representation of $\frg$.

Consider  an $n$-tuple $\vec{\lambda}=(\lambda_1,\dots,\lambda_n)$  of elements in $P_{\ell}(\frg)$. Corresponding to this data, there are vector bundles of conformal blocks
 $\Bbb{V}=\mathbb{V}_{\frg,\vec{\lambda},\ell}$
on the moduli stacks $\overline{\mathcal{M}}_{g,n}$
 of stable $n$-pointed curves of arbitrary genus $g$ \cite{TUY,Fakh} (we will suppress the genus $g$ in the notation of conformal block bundles).
The fiber of $\Bbb{V}$ over $(C,p_1,\dots,p_n)\in \overline{\mathcal{M}}_{g,n}$ is a finite dimensional quotient (after choosing local coordinates at $p_i$) of the form
$${\tensor_{i=1}^n \mh_{\lambda_i,\ell}}/{\frg\tensor \Gamma(C-\{p_1,\dots,p_n),\mathcal{O})} {\tensor_{i=1}^n \mh_{\lambda_i,\ell}}$$
where $\frg\tensor \Gamma(C-\{p_1,\dots,p_n),\mathcal{O})$ acts on ${\tensor_{i=1}^n \mh_{\lambda_i,\ell}}$ via an embedding into $\oplus_{i=1}^n\hat{\frg}$ (see \cite{TUY}).

\subsection{Moduli stacks of parabolic bundles}
\begin{defi}
   Consider an $n$-pointed reduced projective curve $(C;\vec{p})=(C,p_1,\dots,p_n)$, where $p_1,\dots,p_n$ are distinct smooth points of $C$. Let $\Parbun_G(C; \vec{p})=\Parbun_G(C,p_1,\dots,p_n)$ denote the moduli stack parameterizing tuples $(E,\tau_1,\dots,\tau_n)$ where  $E$ is a principal $G$-bundle on $C$, and  $\tau_i\in E_{p_i}/B,\ i=1,\dots,n$.
\end{defi}

Performing this construction in families of stable $n$-pointed curves we  obtain stacks $\Parbun_{G,g,n}$ (in the fppf topology) with morphisms
$$\pi:\Parbun_{G,g,n}\to \overline{\mathcal{M}}_{g,n}$$
such that $\pi^{-1}(C;\vec{p})=\Parbun_G(C; \vec{p})$. Since
$\Parbun_{G,g,n}\to \Parbun_{G,g,0}=\Bun_{G,g}$ is a representable
morphism with smooth fibers, and $\Bun_{G,g}$ is an Artin stack smooth
over $\overline{\mathcal{M}}_{g,n}$ (see e.g., \cite{wang}), it
follows that $\Parbun_{G,g,n}$ is also an an Artin stack smooth over
$\overline{\mathcal{M}}_{g,n}$.


Our aim in this section is to construct a line bundle
$\mathcal{L}=\mathcal{L}_{\frg,\vec{\lambda},\ell}$ on
$\Parbun_{G,g,n}$, and to construct an isomorphism
\begin{equation}\label{isom}
\pi_*{\ml}\leto{\sim}\Bbb{V}_{\frg,\lambda,\ell}^*
\end{equation}
In fact, we will construct such line bundles and isomorphisms
(compatibly with pull backs) for arbitrary families of stable
$n$-pointed curves.


Over $\mathcal{M}_{g,1}$ and $\lambda_1=0$, such line bundles and isomorphisms were given by Laszlo \cite[Section 5]{Laszlo}, building upon the case of a single curve by many authors. Sorger's construction \cite{sorger2} of the line bundle via the theory of conformal blocks plays an important role in Laszlo's construction. Our approach to \eqref{isom} generalizes Laszlo's work \cite{Laszlo}.

\begin{remark}
The functor $\pi_*$ is constructed using \cite[Page 108]{Laumon} and \cite{olly}.
\end{remark}

\subsection{Construction of the line bundle}

Fix a family of stable $n$-pointed curves $\cur\to S$ with the pointed
sections denoted by $\sigma_i:S\to \cur$.  Localizing in the \'{e}tale
topology on $S$, we will make a further choice of disjoint sections
$\tau_{1},\dots,\tau_{m}$ disjoint from $\sigma_1,\dots,\sigma_n$ such
that
\begin{itemize}
\item $\cur'=\cur-\cup_{i=1}^{m}\tau_i(S)-\cup_{j=1}^n\sigma_j$ is affine over $S$,
\item The pointed curves in the family have no  automorphisms.
\end{itemize}

We will construct a line bundle $\ml$ on $\Parbun_S=(\Parbun_{G,g,n})_S$ the stack of parabolic bundles over $\cur\to S$, and show that the line bundle is independent (with canonical isomorphisms) of the choice of the extra sections $\tau_1,\dots,\tau_m$, allowing us (by descent) to define a line bundle without assuming the existence of such sections $\tau_1,\dots,\tau_m$. Set $\sigma_{n+a}=\tau_a$ for $a=1,\dots,m$.

Assume first that $S$ is the spectrum of a ring $R$. For an
$R$-algebra $A$, let $\cur_A=\cur\times_R \operatorname{Spec}(A)$ and
$\cur'_A= \cur'\times_R \operatorname{Spec}(A)$. We have an $R$-group
$L^G_{\cur'}$ (see \cite{BL1,Drinfeld}) which assigns to every
$R$-algebra $A$, the group
$$L^G_{\cur'}(A)=G(\Gamma(\cur'_A,\mathcal{O}))=\operatorname{Mor}_{k}(\cur'_A,G)$$

We also have ind-schemes $\Iwa\subseteq \LoG$ over $S$, whose
$A$-points are defined as follows: Let
$\sigma_1^A,\dots, \sigma^A_{n+m}$ be sections of $\cur_A\to \Spec(A)$
obtained by base change.  Let $\widehat{\cur}_A$ be the completion of
$\cur_A\to \Spec(A)$ along the union of these sections, and
$\widehat{\cur}'_A$, the complement of the union of the induced
sections of $\widehat{\cur}_A\to \Spec(A)$.
\begin{itemize}
\item $\LoG(A)=G(\Gamma(\widehat{\cur}'_A,\mathcal{O}))$,
\item The Iwahori group $\Iwa(A)$ is defined to be subgroup of
  elements of $G(\Gamma(\widehat{\cur}_A,\mathcal{O}))$, which map to
  points of $B(A)^{n+m}$ under the natural map
  $G(\Gamma(\widehat{\cur}_A,\mathcal{O}))\to G(A)^n$.
\end{itemize}
 If local coordinates $z_i$ are chosen along $\sigma_i$, we may identify $\LoG(A)$, and $\Iwa(A)$ with $\prod_{i=1}^{n+m}G(A((z_i)))$ and $\prod_{i=1}^{n+m}I(A)$ where $I(A)$ is the inverse image of $B(A)$ under the morphism $ G(A[[z]])\to G$.

 \begin{proposition}
The $R$-group $L^G_{\cur'}$ is (relatively) ind-affine, formally smooth with connected integral geometric fibers over $S=\operatorname{Spec}(R)$.
\end{proposition}
\begin{proof}
We first claim that the coordinate ring of $\cur'$ is a direct limit of (Zariski) locally free $R$-modules $$\varinjlim {M}_i,\ \  i=1,2,\dots..$$
with ${M}_i$ a local direct summand of ${M}_{i+1}$.

 It suffices to prove this claim  in the universal case where we have a flat family of pointed curves over an integral variety (since the family consists of rigid pointed curves). In this case the coordinate ring of $\cur'$ can be filtered by order of pole along the total boundary divisor (which is ample), and then  standard cohomology and base change arguments give the desired direct limit (we could instead have used a theorem of Lazard that flat modules over any ring are filtered limits of finitely generated free modules \cite{Lazard}).

The argument for the case of a field (which embeds $G$ into a general linear group) generalizes, to give the ind-affine property: The matrix coefficients of the $R$-points of $L^G_{\cur'}$ in the embedding in a general linear group are filtered by the filtration of the coordinate ring. The equations defining $G$ in the general linear group now give the desired ind-affine structure.

For formal smoothness we need to show that $L^G_{\cur'}(A/I)\to L^G_{\cur'}(A)$ is surjective whenever $A$ is an $R$-algebra with a nilpotent ideal $I$. Pick  $\phi: \cur'_{A/I}\to G$ in $L^G_{\cur'}(A/I)$. Now $\cur'_{A/I}\to {\cur'}_{A}$ is a morphism $\Spec(B')\to \Spec(B)$ with $B'=B/J$ with $J\subset B$ nilpotent. Since $G$ is a smooth scheme over the base field, the desired surjection follows.

The assertion on integrality of geometric fibers follows from Proposition \ref{LSgen}.
\end{proof}
\subsection{Central extensions}

Using results of Faltings (\cite[Lemma 8.3]{BL1}, also \cite{LS}), $\LoG$ has a projective representation (this construction is coordinate independent) on
$$\mathcal{H}=R\tensor_k(\tensor_{i=1}^{n+m}\mathcal{H}_{\lambda_i,\ell})$$
lifting a natural representation of affine Kac-Moody algebras in the following sense: Suppose $A$ is an $R$-algebra, hence a $k$-algebra, and $\gamma\in \LoG(A)$. Then locally (in the Zariski topology) on $\Spec(A)$, there is an automorphism
$u_{\gamma}$ of $\mathcal{H}_A=\mathcal{H}\tensor_R A$, unique up to units $R^{\times}$, such that for all $\alpha$ in the $A$-valued points of the corresponding Kac-Moody Lie algebra (which if local coordinates are chosen is $Ac\oplus \oplus_i \frg\tensor A((z_i))$), the following diagram commutes
\begin{equation}\label{proj}
\xymatrix{
\mathcal{H}_A\ar[r]^{\alpha}\ar[d]^{u_{\gamma}} & \mathcal{H}_A \ar[d]^{u_{\gamma}} \\
 \mathcal{H}_A\ar[r]^{\operatorname{Ad}(\gamma)\cdot \alpha}  & \mathcal{H}_A }
\end{equation}
This gives rise to a representation
$$\LoG\to\operatorname{PGL}(\mathcal{H})$$
whose derivative coincides with the natural action of the Lie algebra of $\LoG$ on $\mathcal{H}$ (up to scalars).

Let $\widehat{\LoG}$ be the corresponding central extension of $\LoG$.
\begin{equation}\label{cex}
1\to \Bbb{G}_m\to \widehat{\LoG}\to \LoG\to 1.
\end{equation}
The extension \eqref{cex} splits over $\Iwa\subseteq \LoG$ because the
projective ambiguity disappears over $\Iwa$ (the action over the
tensor product of highest weight vectors can be normalized, see
\cite[Lemma 7.3.5]{sorger3}). We require $\Iwa$ to act on the tensor
product of highest weight vectors via the map to $B(A)^{n+m}$ via the
the product of characters $\lambda_i:B\to \Bbb{G}_m$. Therefore $\Iwa$
is a subgroup of $\widehat{\LoG}$.

Let $\mathcal{I}_G=G(k((z)))/I$ be the Iwahori Grassmannian. It is
known (\cite{Math}, also \cite{KNR} and \cite[Proposition 10.1]{PR})
that the Picard group of $\mathcal{I}_G$ is a direct sum
$\Bbb{Z}\oplus\Pic(G/B)$ where the first factor is generated by line
bundles from the affine Grasmannian, and the second factor comes from
characters of $B$. The $\Bbb{Z}$-component of an element of the Picard
group will be referred to as the level of the line bundle.
\begin{remark}\label{independence}
 The Lie algebra of $\widehat{\LoG}$ is (canonically) isomorphic to the direct sum of affine Lie algebras $\hat{\frg}$ modulo the sum of sum of central elements (considered in families, in a coordinate free way), see \cite[Proposition 4.6]{Laszlo}.

We remark that $\widehat{\LoG}$ is itself independent of the choice of $\lambda_i$ (at the same level). The main point is that the Iwahori Grassmannian $\widehat{\LoG}/\widehat{\Iwa}=\LoG/\Iwa= \mathcal{I}_G^{n+m}\times_k S$ is independent of choice of $\lambda_i$. To see this note that $\widehat{\LoG}$ is a Mumford group (with given isomorphisms) for all line bundles (trivialized over the identity coset) on $\LoG/\Iwa= \mathcal{I}_G^{n+m}\tensor_k S$ with all components at the given level $\ell$ (following \cite[Section 8.3]{LS}), lifting the natural action of $\LoG$ on $\LoG/\Iwa$.


\end{remark}

The following generalizes a result of Sorger \cite{sorger2}:
\begin{lemma}
The extension \eqref{cex} splits over $L^G_{\cur'}\subseteq \LoG$
\end{lemma}
\begin{proof}
 Let $\Bbb{V}=\Bbb{V}_{\frg,\vec{\lambda},\ell}$ be the sheaf of conformal blocks at level $\ell$ on $S$ associated to the data $(\lambda_1,\dots,\lambda_n,0^m)$. Recall that $\Bbb{V}$ is a quotient of $\mathcal{H}$. By Remark \ref{independence}, changing $\lambda_i$ at a fixed level $\ell$ does not change $\widehat{\LoG}$, therefore we may assume $\Bbb{V}\neq 0$  (since this is true if $\lambda_1=\dots=\lambda_n=0$).

 Let
 $$1\to \Bbb{G}_m\to \op{GL}(\Bbb{V})\to \op{PGL}(\Bbb{V})\to 1$$
 be the corresponding central extension.

The projective representation of $L^G_{\cur'}$ on $\mathcal{H}$ passes to a projective representation on $\Bbb{V}$ :
\begin{equation}\label{ek}
\Bbb{V}\tensor_R A =\mathcal{H}_A/ \op{Lie}(L^G_{\cur'})(A)\tensor_A \mathcal{H}_A
\end{equation}
since
\begin{equation}\label{dho}
\op{Lie}(L^G_{\cur'})(A)= \frg\tensor_k \Gamma(\cur'_A,\mathcal{O})
\end{equation}
and take $\gamma\in L^G_{\cur'}(A)$, and $\alpha\in \frg\tensor \Gamma(\cur'_{A},\mathcal{O})=\op{Lie}(L^G_{\cur'})(A)$ in the diagram \eqref{proj}.

The derivative of the projective representation of $L^G_{\cur'}$ on
$\Bbb{V}$ is zero. This is because the map
$\LoG\to\operatorname{PGL}(\mathcal{H})$ has derivative which is up to
scalars given by the natural action of the affine Kac-Moody algebra on
$\mathcal{H}$. The induced morphism $L^G_{\cur'}\to \op{PGL}(\Bbb{V})$
is therefore trivial (Proposition \ref{deriv}). This allows us to fix
the projective ambiguity in the action of $L^G_{\cur'}$ on
$\mathcal{H}$ so that the corresponding action on $\Bbb{V}$ is
actually trivial.
\end{proof}

\subsection{Line bundles}
Set $\widehat{\Iwa}=\Iwa\times \Bbb{G}_m$. The character
$\widehat{\Iwa}=\Iwa\times \Bbb{G}_m\to\Bbb{G}_m$ given by inverse on
$\Bbb{G}_m$ and inverse of product of characters $\lambda_i$ (and
$0^m$) on  $\Iwa$, produces a line bundle on the quotient stack
$$\widetilde{\mq}=\widehat{\LoG}/\widehat{\Iwa}=\LoG/\Iwa.$$
It is easy to see that $\widetilde{\mq}$ is a fiber bundle over  a $(n+m)$-fold product of the affine Grassmannian $\mq_G\times_k S$ (with fibers $G/B)^{n+m}$): In particular, it is a formally smooth ind-scheme over $S$.

Generalizing Proposition \ref{uniformo}, we have
\begin{proposition}\label{globale}
The stack quotient
$$L^G_{\cur'}\setminus \widetilde{\mq}$$ is the pull back of the stack $\Parbun_{G,g,n+m}$ to $S$.
\end{proposition}
The proof again uses \cite{BL2} and  Theorem \ref{triv3}. In particular, we obtain a line bundle $\ml$ on the pull back of $\Parbun_{G,g,n+m}$ to $S$.

We can also identify $\widetilde{\mq}$ with an $(n+m)$-fold product of the Iwahori Grassmannian, and the line bundle $\ml$ pulls back to the line bundle corresponding to $\lambda_i$ at level $\ell$ on this Iwahori Grassmannian. This leads to
(by work of S. Kumar \cite{Kumar}, and Mathieu \cite{Math})
$$H^0(\widetilde{\mq},\ml)=\mathcal{H}^*.$$

In particular, we get using Propositions \ref{LSgen},\ref{globale}, and \ref{Dec30} (and equations \eqref{ek}, and \eqref{dho}) that
\begin{equation}\label{isomm}
H^0((\Parbun_{G,g,n+m})_S,\ml)=\Bbb{V}^*.
\end{equation}
\begin{remark}\label{normalize}
  The line bundles obtained above on $\widetilde{\mq}$ are trivialized
  along the identity coset. Similarly, the line bundle on
  $\Parbun_{G,g,n+m}$ is trivialized along the trivial section of
  $\Parbun_{G,g,n+m}$ over $S$. Such rigidifications fix the line
  bundle up to canonical isomorphisms.
\end{remark}
Adding one more section $\tau$ (after further affine \'{e}tale localization $S'$ of $S$) leads to a line bundle $\ml'$ on $\Parbun_{G,g,n+m+1}$. It it is easy to see (using propagation of vacua isomorphisms) that the pull back of $\ml'$ to $L^G_{\cur'}\setminus \widetilde{\mq}$ (pulled up to $S'$) is identified canonically with $\ml$.
\subsection{Proof of Theorem 1.7}
We now use a descent technique from \cite{Fakh} (see the discussion following Proposition 2.1 in \cite{Fakh}) to show the independence of $\mathcal{L}$ from the choice of sections $\tau_1,\dots,\tau_m$. Suppose we are given two sets of the $\tau$'s, say $ \tau_1,\dots,\tau_m$ and $\tau'_1,\dots,\tau'_{m'}$. Add a further large collection of sections $\sigma_1,\dots,\sigma_{s}$ to each of these sets disjoint from the earlier sections (after \'{e}tale localization in S). We have isomorphisms of the $\ml$ coming from $\tau_1,\dots,\tau_m$ with that coming from  $\tau_1,\dots,\tau_m,\sigma_1,\dots,\sigma_{s}$, which is in turn canonically isomorphic to to the line bundle coming from $\sigma_1,\dots,\sigma_{s}$, and similarly for $\tau'_1,\dots,\tau'_{m'}$. This gives the patching data (the cocycle condition is easy to check).

Therefore, there exists (by descent) a line bundle $\ml$ on the stack
$\Parbun_{G,g,n}$. It is now easy to see that the desired isomorphism
\eqref{isom} holds. This completes the proof of Theorem
\ref{confreal}.

It also follows that the line bundle $\ml$ can be defined for
arbitrary bases $S$. To prove the isomorphism \eqref{isom} over a
base, we may localise and assume $S$ to be affine, and apply
\eqref{isomm} when $S$ is finite type over $k$ (the finite type
requirement comes from Proposition \eqref{Dec30}) .
\begin{remark}
We may show that \eqref{isomm} holds for arbitrary affine $S$ (not necessarily of finite type over $k$) as follows. The implication (a) implies (b) in Proposition \eqref{Dec30} works for arbitrary affine $S$ snd hence there is an injective map \eqref{isomm} for arbitrary $S$. For the surjection, we can assume since the sheaves of covacua $\Bbb{V}^*$ are locally free of finite rank  (and commute with base change), that there is an $R$-basis of the sheaf of vacua coming from a family over a finitely generated base. Therefore any particular element of   $\Bbb{V}^*$ is a the pull back from an  affine scheme $S'$ of finite type over $k$ via a map $S\to S'$. We can now apply the isomorphism \eqref{isomm} for $S'$ and the pull back map $(\Parbun_{G,g,n+m})_{S}\to (\Parbun_{G,g,n+m})_{S'}$ to conclude that \eqref{isomm} is surjective for all schemes $S$ over $k$.
\end{remark}

\subsection{Explicit description of the line bundles $\ml$ on $\Parbun_{G,g,n}$}
Let $V$ be an irreducible representation of $G$ with Dynkin index $d_V$ (see \cite{KNR}).
Consider the line bundle $\ml'$ on $\Parbun_{G,g,n+m}$ which over a marked curve $(C,\vec{p})$  and point $(E,\tau_1,\dots,\tau_n)\in \Parbun_G(C; \vec{p})$
 equals the tensor product of
\begin{itemize}
\item  The line $\bigl(\det(H^*(C,E\times_G V))\tensor \det(H^*(C,V\tensor \mathcal{O}))^{-1}\bigr)^{\ell}$. Here $\det(H^*(C,W))$ for a coherent sheaf $W$ on $C$ is the line
     \begin{equation}\label{DCoh}
    \bigwedge^{\op{top}}H^0(C,W)^*\tensor \bigwedge^{\op{top}} H^1(C,W)
    \end{equation}
\item The lines (for $i=1,\dots,n$) obtained as  fibers of  $E_{p_i}\times_B \Bbb{C}_{-\lambda_i}\to E_{p_i}/B$
over the given elements $\tau_i$.
\end{itemize}

It is easy to see (using Remark \ref{normalize}, and Lemma \ref{liszt}) that $\ml'$ equals the line bundle $\ml$ on $\Parbun_{G,g,n+m}$ associated to the data $\vec{\lambda}$ but at level $\ell d_V$. Therefore Theorem \ref{confreal} as a special case gives
\begin{theorem}\label{confprime}
$\pi_*{\mathcal{L}'} =\Bbb{V}^*_{\frg,\vec{\lambda},\ell d_V}$
\end{theorem}
Similar to Theorem \ref{confreal}, the version for arbitrary families of $n$-pointed curves also holds.

  When $G$ is a special linear or symplectic group, and $V$ is the standard representation, $d_V=1$. The standard (vector) representation for Spin groups has $d_V=2$, and therefore odd levels are not covered by Theorem \ref{confprime} (but covered by Theorem \ref{confreal}).

The theory of Pfaffian line bundles \cite {LS} for spin  groups and fixed smooth curves to cover odd levels seems to run into difficulties  for singular curves for degree reasons: the degree of the dualizing sheaf of a reducible curve may be odd on irreducible components.

\subsection{Generalities on ind-group actions}\label{gammastuff}
Let $S=\Spec{R}$ and suppose $\Gamma$ is an ind-group over $S$ acting on an ind-scheme $\mathcal{Q}$ over $S$. Let $\mathcal{L}$ be a $\Gamma$-linearized line bundle on $\mathcal{Q}$. In this setting $\op{Lie}(\Gamma)(R)=\ker(\Gamma(R[\epsilon]/\epsilon ^2)\to \Gamma(R))$ acts on $H^0(\mathcal{Q},\mathcal{L})$ (note that since $R'=R[\epsilon]/(\epsilon ^2)$ is a free $R$-algebra, $H^0(\mathcal{Q}_{R'},\mathcal{L}_{R'})=H^0(\mathcal{Q},\mathcal{L})\tensor_R R'$).

We will make the following assumptions in this section
\begin{itemize}
\item Assume that $\Gamma$ and $\mathcal{Q}$ are inductive limits of
  schemes of finite type over $S$, and that $S$ is a scheme of finite
  type over $k=\Bbb{C}$. In addition assume that $\Gamma$ is
  ind-affine and formally smooth over $S$ with integral geometric
  fibers,
\end{itemize}

\begin{defi}
A section $s\in H^0(\mathcal{Q},\mathcal{L})$ is $\Gamma$-invariant if $\Gamma(A)$ acts trivially on the pullback section $s_A$ of $H^0(\mathcal{Q}_A,\ml_A)$ for all $R$-algebras $A$.
This definition can be reformulated as the vanishing of a suitable section of a the pull back of $\ml$ on $\Gamma\times_S \mathcal{Q}$ \cite[Lemma 7.2]{BL1}.

By  \cite[Proposition 7.2]{BL1}, the space of such sections is in bijection with $H^0(\Gamma\setminus \mathcal{Q},\ml_0)$ where $\ml_0$ is the line bundle on the stack $\Gamma\setminus\mathcal{Q}$ obtained via descent from $\ml$. Note that in this generality $\Gamma\setminus\mathcal{Q}$ may not be algebraic.
\end{defi}
The following is a generalization of \cite[Proposition 7.4]{BL1}
\begin{proposition}\label{Dec30}
Let $s\in H^0(\mathcal{Q},\mathcal{L})$.  The following are equivalent.
\begin{enumerate}
\item[(a)] The section $s$ is $\Gamma$-invariant.
\item[(b)] For every $R$-algebra $A$, the pull-back section $s_A$ of
$H^0(\mathcal{Q}_A,\mathcal{L}_A)$ is invariant under the Lie algebra $\operatorname{Lie}(\Gamma)(A)$ (which is the kernel of $\Gamma(A[\epsilon]/(\epsilon^2))\to \Gamma(A)$).
\end{enumerate}
\end{proposition}
\begin{proof}
The proof in \cite{BL1} carries over easily to give that (a) implies (b). For the reverse direction we adapt the arguments of \cite{BL1} as follows.

Let $m:\Gamma\times_S \mathcal{Q}\to \mathcal{Q}$ be the group action
and $pr_2: \Gamma\times_S\mathcal{Q}\to \mathcal{Q}$ the
projection. We also have a given isomorphism $\phi:m^*{\mathcal{L}}\to
pr_2^*\mathcal{L}$. Let $\sigma =\phi(m^*s)-pr_2^*s$. Assume (b); to prove (a) we need to show that $\sigma=0$.

Suppose $\sigma\neq 0$. Write $\Gamma$ and $\mathcal{Q}$ as ind-schemes over $S$:
$$\Gamma=\lim_{\rightarrow} \Gamma^{(N)},\ \mathcal{Q}=\lim_{\rightarrow} \mathcal{Q}^{(M)}.$$
Here $M,N$ vary over the natural numbers, and $\Gamma^{(N)},\mathcal{Q}^{(M)}$ are schemes of finite type
over $S$.

Assume that $\sigma$ is non-zero on $\Gamma^{(N_0)}\times_S \mathcal{Q}^{(M_0)}$. Base changing to $T=Q^{(M_0)}$ we have an ind-scheme $\Gamma_T=\lim_{\rightarrow} \Gamma^{(N)}\times_S T=\lim_{\rightarrow} \Gamma_T^{(N)}$. The restriction of $\sigma$ to $\Gamma_T$ is clearly non-zero. 

Let $x\in \Gamma^{(N_0)}_T$ be a closed point of $\Gamma_T$ such that $\sigma$ is non-zero in the stalk at $x$ in $\Gamma^{(N_0)}_T$. Let $y\in T$ be the image. For every positive integer $u$ we can form the base change diagram
\begin{equation}
\xymatrix{
\Gamma_T\times_T T_u \ar[r]\ar[d] & \Gamma_T \ar[d] \\
  T_u =\Spec{\mathcal{O}_{T,y}/m_y^u}\ar[r]  & T }
\end{equation}

Let $u\geq 0$ be the smallest integer so that the image of $\sigma$ is
non-zero in $\Gamma_T\times_T T_{u+1}$ (since $\sigma$ restricted to
some $\Gamma_T^{(N)}$ is non-zero, such  a $u$ exists).

Now there is a map
$$H^0((\Gamma_T)_y,\mathcal{O}_{(\Gamma_T)_y})\tensor_k\mathcal{L}_y\tensor_k m_{T,y}^{u}/m_{T,y}^{u+1}\to H^0(\Gamma_T\times_T T_{u+1},\mathcal{L})$$
We claim that this map is injective and $\sigma$ is in the image. The local situation is the following:
\begin{enumerate}
\item Set $T_{u+1}=\Spec(A)$ where $(A,m)$ is a local ring and $m^{u+1}=0$, $k=A/m$ (and $A$ is a $k$-algebra).
\item Write $\Gamma_A$  as an inductive limit of $\Spec(B_j)$, set $B=\varprojlim B_j$ where $B_j$ are $A$-algebras.
\item The rings $B_j/m_jB_j$ are of finite type over $k$.
\item There is a section of $\Gamma_A\to \Spec(A)$ passing through any closed  point of the central fiber (using the  formal smoothness of $\Gamma$).
\item Pick a $k$-basis $a_1,\dots,a_r$ of $m^u$.
\item Suppose $f\in \varprojlim_j m^u B_j \subseteq \varprojlim B_j$.
\end{enumerate}
Our aim is then to show that
\begin{itemize}
\item We can write $f$ uniquely as $f=\sum_{i=1}^r a_i g_i$ where $g_i\in \varprojlim B_j/mB_j$.
\end{itemize}
We can find such lifts $g_i$ at the level of each $B_j$, non uniquely. The ambiguity is the module of relations
$$M_j=\{(h_1,\dots, h_r)\in (B_j/mB_j)^{\oplus r}\mid \sum a_i h_i=0\in m^u B_j\}.$$
Therefore there are exact sequences
$$0\to M_j\to (B_j/mB_j)^{\oplus r}\to m^uB_j\to 0.$$

It suffices now to show (using Mittag-Leffler properties, see e.g., \cite[Proposition II.9.1]{Hartshorne}) that for any $j$, there is a $n>j$ such that $M_n\to M_j$ is the zero map. This will also
lead to an isomorphism (recall that $m^{u+1}=0$)
\begin{equation}\label{clef}
m^u\tensor_k \varprojlim B_i/mB_i\leto{\sim}\varprojlim_i m^u B_i \subseteq \varprojlim B_i
\end{equation}

First we show that there is an $\ell>j$ such that any nilpotent
element in $B_{\ell}/m B_{\ell}$ goes to to $0$ in $B_j/mB_j$. This is
because the closed fiber of $\Gamma_A$ is the limit of an inductive system $Y_c$ of
reduced varieties. Therefore, there exists $\ell > j$ so that we have
a factoring
$\Spec(B_{j}/m B_{j})\to Y_c \to \Spec(B_{\ell}/m B_{\ell})$. Any
nilpotent element in $B_{\ell}/mB_{\ell}$ must vanish on $Y_c$ so goes
to $0$ in $B_j/mB_j$.

Fix such an $\ell$ and let $Z_{\ell}=\Spec(B_{\ell}/m B_{\ell})$. The scheme $Z_{\ell}\times_k \Spec(A)$ has a map to $\Gamma$ (over $\Spec(A))$ lifting the map of $Z_{\ell}$ to the closed fiber of $\Gamma_A$ (using formal smoothness). Pick $n$ such that $Z_{\ell}\times_k \Spec(A)\to \Gamma$ factors through $\Spec(B_{n})$. This $n$ ``works": given an equation $\sum a_i h_i=0\in m^u B_n$ with $h_i \in B_n/mB_n$, we can restrict to the sections of $\Spec(B_{n})\to \Spec(A)$ coming from the map $Z_{\ell}\times_k \Spec(A)\to \Spec(B_n)$. Since $a_1,\dots, a_r$ are $k$-linearly independent, $h_i$ vanish at all points in the image of $Z_{\ell}$\footnote{Write $h_i$ as a sum of a constant and a function which vanishes at a given point of the closed fiber to facilitate pull back to $\Spec(A)$ via sections.},  (i.e., give a nilpotent function on $Z_{\ell}$), and hence map to $0\in B_j/mB_j$ as desired.


Therefore we get a non-zero
$$\bar{\sigma}\in
H^0((\Gamma_T)_y,\mathcal{O}_{(\Gamma_T)_y})\tensor_k V, \ \
V=\mathcal{L}_y \tensor_k m_{T,y}^{u}/m_{T,y}^{u+1} .$$
(We define $m_{T,y}^{0}$ to be the local ring of $T$ at $y$.)

Now $\bar{\sigma}$ is a section of the trivial vector bundle on
$(\Gamma_T)_y$ with fibers $V$ and it suffices to show that
$\bar{\sigma}$ is constant (i.e., given by an element of $V$): The
constant has to be zero since $\sigma$ is zero when pulled back via
the identity section $\mathcal{Q}\to \Gamma\times_S \mathcal{Q}$.

Let $A=\mathcal{O}_{T,y}/m_{T,y}^{u+1}$. Assume that $\bar{\sigma}$ is
not zero on the integral ind-scheme $(\Gamma_T)_y$. Write $V=k^{r}$
for some $r$. Then there is a tangent vector
$X:\Spec(k[\epsilon]/(\epsilon^2))\to (\Gamma_T)_y$ which does not
kill one of the $r$ components of $\bar{\sigma}$. Since $\Gamma_T$ is
formally smooth over $T$, we can extend $X$ to a tangent
$X_A:\Spec(A[\epsilon]/\epsilon^2)\to \Gamma_T$ over $T$.  Now $X_A$
gives rise to a map $Y_A:\op{Spec}(A)\to \Gamma_T$ obtained as a
composite
$$ \op{Spec}(A)\to \op{Spec}(A[\epsilon]/\epsilon^2)\leto{X_A} \Gamma_T.$$
Set $X'_A$ to be the composite
 $$\Spec(A[\epsilon]/\epsilon^2)\to \op{Spec}(A)\leto{Y_A} \Gamma_T.$$

We have a tautologous $A$-point $p$ of $\mq_A$. Let $q$ denote the point $p$ translated by $Y_A$. Consider the two elements in $\ml_p\tensor_A A[\epsilon]/(\epsilon^2)$  obtained by pulling back $\sigma$ under $X_A$ and $X'_A$ respectively. By our assumption, the difference $\epsilon\Delta$ of these elements is non-zero. We therefore get a non-zero element $\Delta\in \ml_p$.

Set $\alpha=X_A\circ (X'_A)^{-1}\in \op{Lie}(\Gamma)(A)$ and we obtain $\alpha \sigma_A\in H^0(\mq_A,\ml_A)$. It is easy to see that $\Delta\in \ml_p$  agrees with the value of $\alpha\sigma_A$ at $q$ (an element of $\ml_q$)  under the  isomorphism of $A$-modules $\ml_q\leto{\sim}\ml_q$ given by $X'_A$.

Since by assumption (b) $\alpha \sigma_A=0$ and $\Delta\neq 0$, we reach a contradiction.
\end{proof}

\begin{proposition}\label{deriv}
Suppose  $\Gamma\to S=\Spec(R)$ has integral fibers, where $S$ is a scheme of finite type over $k=\Bbb{C}$. Let $\phi:\Gamma\to \operatorname{Aut}(V)$ where $V$ is a finite dimensional vector bundle over $S$, be a map of $R$-groups. Assume that for every $R$-algebra $A$, $\operatorname{Lie}(\Gamma)(A)=\ker \Gamma(A[\epsilon]/\epsilon^2)\to \Gamma(A)$ maps to $0\in \operatorname{End}(V)\tensor_ R A$ under $\phi$. Then $\phi$ is the trivial morphism.
\end{proposition}
\begin{proof}
  We may assume by passing to an affine cover of $S$ that $V$ is
  trivial. Assume now that $\phi$ is not trivial, and let $y\in S$ and
  $u\geq 0$ be such that the image of $\phi$ is not the identity in
  $\Gamma\times_S S_{u+1}$, where
  $S_{u+1} =\Spec{\mathcal{O}_{S,y}/m_y^{u+1}}$, with $u$ the smallest
  possible. Base changing the picture to $S_u$, we get a matrix of
  functions $(f_{a,b})$ on $\Gamma$. Therefore $(f_{a,b})$ differs
  from the identity matrix by a matrix $G=(g_{ab})$ with $g_{a,b}\in m^u B_j$ for
  every $j$, where we have expressed $\Gamma_{S_{u+1}}$ as
  $\varprojlim B_j$.

The matrix $G$ can be lifted to a matrix with coefficients in $\varprojlim m^u\tensor_k B_j/mB_j$ (as before, see \eqref{clef}), and our aim is to show that  elements of $\varprojlim  B_j/mB_j$ coming from matrix coefficients (after choosing a basis of $m^u$) have zero derivatives. The rest of the proof is similar to that of Proposition \ref{Dec30}, using the group action to move sections of $\Gamma$ to the identity section.
\end{proof}

\section{Picard Groups}\label{c2}
Fix an $n$ pointed-curve $(C;\vec{p})$ with arbitrary singularities. Let $m$ be the number  of irreducible components of $C$. Choose smooth points $q_1,\dots,q_m\in C$ distinct from $p_1,\dots,p_n$, and local uniformizing parameters at the points $p_1,\dots,p_n$.

 We have a morphism of stacks  $\mathcal{Q}_G^m\times (G/B)^n\to \Parbun_G(C; \vec{p})$, which by Proposition \ref{uniformo} induces an isomorphism
$$L_{C-\{q_1,\dots,q_m\}}(G)\setminus \mathcal{Q}_G^m\times (G/B)^n\to \Parbun_G(C; \vec{p}).$$

Recall that $\Pic\mathcal{Q}_G=\Bbb{Z}$.
\begin{proposition}\label{pb}
The natural pull back morphism
\begin{equation}
\Pic(\Parbun_G(C;\vec{p}))\to \Pic(\mathcal{Q}_G^m\times (G/B)^n)=\Bbb{Z}^m\oplus \oplus_{i=1}^n \Pic(G/B)
\end{equation}
is an isomorphism.
\end{proposition}
\begin{proof}
By Proposition \ref{uniformo}, $\Pic(\Parbun_G(C;\vec{p}))$ coincides with the group of $L_{C-\{q_1,\dots,q_m\}}(G)$-linearized bundles on
$\mathcal{Q}_G^m\times (G/B)^n$. By Proposition \ref{LSgen}, and \cite[Corollary 5.1]{LS}, there can be at most one such linearization of a line bundle on
$\mathcal{Q}_G^m\times (G/B)^n$. This shows that the morphism in \eqref{pb} is injective.

To show the surjectivity, we use the line bundles of form $E\mapsto E_{p_i}\times_B \Bbb{C}$ where $B$ acts on $\Bbb{C}$ via a character $\chi$ to produce line bundles of the form $(0^m,0,\dots,\chi_i,\dots,0)$ with $\chi$ in the $i$th place. This takes care of all the $G/B$ contributions. For the remaining contributions we use diagrams induced by restricting bundles to irreducible components. If $q_j$ is on the irreducible component $C_j$, consider
\begin{equation}
\xymatrix{
\mathcal{Q}_G^{m}\times (G/B)^n\ar[r]\ar[d] & \mathcal{Q}_G\ar[d]  \\
 \Parbun_G(C;\vec{p})\ar[r]  & \Bun_G(\widetilde{C}_j)
 }
 \end{equation}
where the horizontal map on top is projection to the $j$th factor of $\mathcal{Q}_G^m$, and $\widetilde{C}_j$ the normalization of $C_j$. Here we use the fact that the uniformization map  $\mathcal{Q}_G\to \Bun_G(\widetilde{C}_j)$ corresponding to $(\widetilde{C}_j,q_j)$ is an isomorphism on Picard groups by \cite{Laszlo,sorger2}.
\end{proof}

Let $G\to\operatorname{SL}(V)$ be an irreducible representation of $G$. Associated to $V$ there is a notion of a Dynkin index $d_V$ (cf. \cite{KNR}): It has the following property. If $p$ is smooth point (with local coordinates) on a smooth projective curve $C$, the associated morphism $\mathcal{Q}_G\to \operatorname{Bun}_G(C)$ pulls back  the determinant of cohomology line  bundle $\det(H^*(C,E\times_P V))$ (see \eqref{DCoh}) on $\operatorname{Bun}_G(C)$ to $d_V$ times the positive generator $\mathcal{O}(1)$ of the Picard group of $\mathcal{Q}_G$ (see \cite{LS}).
\begin{lemma}\label{liszt}
Suppose $y_1,\dots,y_m$ are distinct smooth points on a possibly singular projective curve $C$. The corresponding map (after choosing local coordinates at the points $p_i$) $\mathcal{Q}_G^m\to \operatorname{Bun}_G(C)$ pulls back the determinant of cohomology bundle $\det(H^*(C,E\times_P V))$ to $\boxtimes_{i=1}^m\mathcal{O}(d_V)$.
\end{lemma}
\begin{proof}
Since the Picard group of $\mathcal{Q}_G$ is $\Bbb{Z}$, it follows that any line bundle on $\mathcal{Q}_G^m$ is of the form
 $\boxtimes_{i=1}^m\mathcal{O}(a_i)$. It suffices to therefore assume $m=1$. Let $\pi:\widetilde{C}\to C$ be the normalization of
 $C$, and we have maps $\mathcal{Q}_G\to \Bun_G(C)\to \Bun_G(\widetilde{C})$. The statement then comes down to verifying that
 $\det(H^*(\widetilde{C},\pi^*E\times_P V))=\det(H^*(C,E\times_P V))$ for any principal bundle $E$ on $C$ canonically (Note that $\mathcal{Q}_G\to \Bun_G(\widetilde{C})$ produces $G$ bundles on $\widetilde{C}$ which are trivial on
 connected components of $\widetilde{C}$ other than the one mapped to by $y_1$, so we may assume $\widetilde{C}$ is connected.)

  We show that if $W$ is a vector bundle with trivial determinant on $C$, then $\det(H^*(C,\pi_*\pi^* W))=\det(H^*(C,W))$. Let $M=\pi_*\mathcal{O}/\mathcal{O}$. We are reduced to showing that $\det(H^*(C,M\tensor W))$ is trivial. But $M$ is torsion on $C$ and can be filtered so that the graded quotients are (directs sums) of the form $i_*k$ where $i$ is the inclusion of a closed point in $C$. The case $M=i_*k$ is immediate since $W$ has trivial determinant.
\end{proof}

\vspace{0.05 in}

\noindent
P.B.: Department of Mathematics, University of North Carolina, Chapel Hill, NC 27599, USA\\
{{email: belkale@email.unc.edu}}

\vspace{0.08 cm}

\noindent
N.F.: School of Mathematics, Tata Institute of Fundamental Research, Homi Bhabha Road, Mumbai 400005, India\\
{{email: naf@math.tifr.res.in}}
\vspace{0.08 cm}

\end{document}